\documentclass[12pt,notitlepage]{amsart}
\usepackage{latexsym,amsfonts,amssymb,amsmath,amsthm}
\usepackage{mathrsfs}
\usepackage{wrapfig}
\usepackage{stackrel}
\usepackage{color}
\usepackage{float}
\usepackage{enumitem}
\usepackage{mathtools}
\usepackage{multicol}
\usepackage[normalem]{ulem}
\usepackage{verbatim,amscd,graphicx}
\usepackage{graphics}
\usepackage{relsize}
\usepackage{esvect}
\usepackage{multirow}

    \newtheoremstyle{TheoremNum}
        {\topsep}{\topsep}              %%% space between body and thm
        {\itshape}                      %%% Thm body font
        {}                              %%% Indent amount (empty = no indent)
        {\bfseries}                     %%% Thm head font
        {.}                             %%% Punctuation after thm head
        { }                             %%% Space after thm head
        {\thmname{#1}\thmnote{ \bfseries #3}}%%% Thm head spec
    \theoremstyle{TheoremNum}

\usepackage[inner=1.0in,outer=1.0in,bottom=1.0in, top=1.0in]{geometry}

\newcommand{\mymod}{\ensuremath{\negthickspace \negmedspace \pmod}}
\newcommand{\shortmod}{\ensuremath{\negthickspace \negthickspace \negthickspace \pmod}}

\DeclareMathOperator{\Vol}{Vol}
\newcommand{\norm}[1]{\|{#1}\|}
\newcommand{\Tr}{\operatorname{Tr}}
\newcommand{\Supp}{\operatorname{Supp}}

\theoremstyle{plain}       
    \newtheorem{mtheo}{Theorem} [section]
    
    \newtheorem{mprop}[mtheo]{Proposition}
    \newtheorem{mcoro}[mtheo]{Corollary}
     \newtheorem{mlemma}[mtheo]{Lemma}

\theoremstyle{remark}       
    
    \newtheorem{mrema}[mtheo]{Remark}
    
    \newtheorem{mconv}[mtheo]{Convention}

\numberwithin{equation}{section}
\numberwithin{figure}{section}
\begin{document}

\title{Quantum Unique Ergodicity for Eisenstein Series in the Level Aspect}
\author{Jiakun Pan and  Matthew P. Young}
\email{jpan@math.tamu.edu, myoung@math.tamu.edu}
\address{Department of Mathematics\\
	  Texas A\&M University\\
	  College Station\\
	  TX 77843-3368\\
		U.S.A.}

% \author{}
% \email{}
% \address{Department of Mathematics\\
% 	  Texas A\&M University\\
% 	  College Station\\
% 	  TX 77843-3368\\
% 		U.S.A.}
\thanks{This material is based upon work supported by the National Science Foundation under agreement No. DMS-170222 (M.Y.).  Any opinions, findings and conclusions or recommendations expressed in this material are those of the authors and do not necessarily reflect the views of the National Science Foundation.
 }		
		
\begin{abstract}
    We prove a variety of quantum unique ergodicity (QUE) results for 
    Eisenstein series in the level aspect.  A new feature of this variant of QUE is that the main term involves the logarithmic derivative of a Dirichlet $L$-function on the $1$-line. 
    A zero of this $L$-function near the $1$-line can thus have a distorting effect on the main term.
    
    We obtain quantitative control on the test function and thereby prove an asymptotic formula in the level aspect version of the problem with test functions of shrinking support.  Surprisingly, this asymptotic formula shows some obstruction to equidistribution that may retrospectively be interpreted as being caused by the growth of Eisenstein series in the cusps.
    We also make some coarse descriptions on the unevenness of the mass distribution of level $N$ Eisenstein series on the fibers of the canonical projection map from $Y_0(N)$ to $Y_0(1)$.
\end{abstract}
\maketitle
\section{Introduction}
\subsection{Foreword}
Let $\mathcal{M} = \Gamma \backslash \mathbb{H}$ be a hyperbolic manifold of finite volume, and $\{u_j\}$ be the sequence of $L^2$-normalized eigenfunctions of increasing eigenvalues for the Laplace--Beltrami operator $\Delta$. The quantum unique ergodicity (QUE) conjecture of Rudnick and Sarnak \cite{RS} predicts
\begin{align}\label{QUE}
\int_{\mathcal{M}} |u_j|^2 \phi d\mu {\longrightarrow} \frac{1}{\Vol(\mathcal{M})}\int_{\mathcal{M}} \phi d\mu,
\end{align}
for all fixed nice (e.g., continuous and bounded) test functions $\phi$ as $j\rightarrow\infty$. 

The QUE conjecture sparked a lot of work for different families of automorphic forms. 
One of the earliest unconditional QUE results is for the classical Eisenstein series $E_t:= E(z,\frac{1}{2}+it)$ on $\mathcal{M}=SL_2(\mathbb{Z})\backslash\mathbb{H}$ equipped with the Poincar\'e measure $d\mu = y^{-2} dxdy$.  Since $E_t$ is not square integrable, we need to further assume $\phi$ is compactly supported, and in this scenario, Luo and Sarnak \cite{LS} showed as $t \rightarrow\infty$,
\begin{align}\label{LS}
\frac{1}{\log (\frac{1}{4}+t^2)}\int_{\mathcal{M}} |E_t|^2 \phi d\mu {\longrightarrow} \frac{1}{\Vol(\mathcal{M})} \int_{\mathcal{M}} \phi d\mu.
\end{align}
The second author \cite{Y1} estimated the rate of convergence with a power saving bound for the error terms, which allowed the test function $\phi=\phi_t$ to change mildly, e.g., by having shrinking support. In particular, for each fixed point $z\in \mathcal{M}$, (\ref{LS}) holds if $\phi_t$ is the characteristic function of a ball of radius $r=r(t)$ centered at $z$, with $r = t^{-\delta}$, for some $\delta>0$. We refer readers to \cite{Lin}, \cite{So}, and \cite{HS} for some of the significant developments on QUE in either eigenvalue or weight aspect; a survey paper \cite{Sa} by Sarnak is good to begin with.

% Lindenstrauss \cite{Lin} has settled the  original  QUE conjecture, without a rate of convergence, for compact arithmetic hyperbolic surfaces.  Soundararajan \cite{So} added new ideas to extend the proof in order to treat (non-compact) congruence subgroups.  Holowinsky and Soundararajan \cite{HS} have proven the mass equidistribution conjecture, which is the variant of QUE for holomorphic modular forms of large weight.

Kowalski, Michel and VanderKam \cite{KMV} formulated the level aspect analog of QUE. Let $f^{\scriptscriptstyle{(N)}}$ be a sequence of holomorphic newforms of fixed even weight on $Y_0(N)= \Gamma_0(N) \backslash \mathbb{H}$, which are $L^2(Y_0(N))$-normalized with the measure $d\mu$. They conjectured
\begin{align}\label{KMV}
\int_{Y_0(N)} |f^{\scriptscriptstyle{(N)}}|^2 \phi d\mu \longrightarrow \frac{1}{\Vol(Y_0(1))} \int_{Y_0(1)} \phi d\mu,
\end{align}
for fixed $\phi$ of level $1$,
as $N \rightarrow \infty$.  The conjecture is now known due to \cite{N} and \cite{NPS}, which in fact proved QUE in both weight and level aspects.
For the case of Eisenstein series, Koyama \cite[Theorem 1.2]{K} showed 
\begin{align}\label{K}
\frac{1}{2 \log N}\int_{Y_0(N)} |E^{\scriptscriptstyle{(N)}}|^2 \phi d\mu \longrightarrow  \frac{1}{\Vol(Y_0(1))} \int_{Y_0(1)} \phi d\mu,
\end{align}
for fixed $T\in\mathbb R$, as
$N \rightarrow\infty$ traversing all prime numbers, and where $E^{\scriptscriptstyle{(N)}}=E_{\infty}^{\scriptscriptstyle{(N)}}(z,\frac{1}{2}+iT)$ are Eisenstein series of weight zero, level $N$ and trivial central character.

We should clarify that \eqref{K} is perhaps not the closest analog of (\ref{KMV}) for Eisenstein series, because these $E^{\scriptscriptstyle{(N)}}$ of trivial central character are oldforms. The newform Eisenstein series $E_{\chi_1,\chi_2}$ defined in \cite{Y2} should be the perfect counterpart of holomorphic newforms in \cite{KMV}, where $\chi_i$ is primitive mod $q_i$, for $i=1,2$, and $q_1 q_2 =N$, and the equidistribution problems around these Eisenstein series are noteworthy, attractive, and closer in spirit to (\ref{KMV}). As we later argue, a large number of such newforms are actually of the form $E^{\scriptscriptstyle{(N)}}_{\mathfrak{a}}$ with primitive central character, for which QUE is given in Theorem \ref{thirdary} below.

\subsection{First results}
\label{section:firstresults}

\begin{mconv}\label{conv1}
We comply with the following notational conventions throughout this paper.
\begin{itemize}
\item We denote the space of smooth automorphic functions of central character $\chi$ on the manifold $Y_0(N)$ by $\mathcal{A}(Y_0(N), \chi)$. We may suppress $\chi$ if it is trivial.
\item We write $\langle f,g \rangle_{_N}$ by $\int_{Y_0(N)} f\cdot\overline{g}d\mu$, if $f,g\in \mathcal{A}(Y_0(N), \chi)$. So, $\langle 1,1 \rangle_{_N}=\text{Vol}(Y_0(N))$.
\item When $N=1$, we write $\norm{f}_p$ short for $\langle |f|^p,1 \rangle_{_1}^{1/p}$.
\item For a Dirichlet character $\chi \pmod{N}$, we always assume it is induced by primitive $\psi\pmod{q}$, for some $q\mid N$. We regard the character $\mymod{1}$ as primitive. 
\item We let $\theta$ be so that the $p$-th Hecke eigenvalues of Maass newforms are uniformly bounded by $p^{\theta} + p^{-\theta}$.   
The value $\theta = 7/64$ is allowable
by \cite{KS}.
\item When writing $f \ll g$ or $f = O(g)$ we typically add a subscript to denote dependence of the implied constant on ambient parameters; an exception to this is that we do not display this when a constant depends on $\varepsilon$.
\end{itemize}
\end{mconv}

\begin{mtheo}\label{secondary}
Let $E=E_{\infty}(z,s,\chi)$ be the Eisenstein series of level $N$, weight zero and central character $\chi$ (see (\ref{E}) for definition), with $s=\frac{1}{2}+iT$ for fixed real $T\neq 0$. For all compactly supported $\phi_{_0} \in \mathcal{A}(Y_0(1))$, we have
\begin{multline}
\label{eq:QUE}
\langle |E|^2, \phi_{_0} \rangle_{_N}= \frac{\langle 1,\phi_{_0} \rangle_{_1}}{\langle 1,1 \rangle_{_1}} \Big( 2\log N + 4\Re \frac{L'}{L}(1+2iT,\overline{\psi})\Big) \\
+ O_{T, \phi_{_0}} ( (\log\log N)^5 ) + O_{T}( N^{-\frac{1}{8}+\varepsilon}
(\tfrac{N}{q})^{-\frac{3}{8}+\theta} 
\norm{\phi_{_0}}_2 ).
\end{multline}
\end{mtheo}

%Since it is unknown if $4\Re \frac{L'}{L}(1+2iT,\overline{\psi}) = o_T(\log{q})$, we must include it as part of the main term.  Of course, such a bound holds on GRH (see \cite[Theorem 5.17]{IK}).  Unconditionally, this type of bound can be shown for almost all $\psi \pmod{q}$, using zero density estimates, for example.  Also, if $\psi$ is fixed, this logarithmic derivative can be absorbed into the error term.

\begin{mrema}
Here we discuss the case $T=0$. When $q\neq 1$, \eqref{eq:QUE} also holds as $\frac{L'}{L}(s,\overline{\psi})$ is well-defined at $s=1$; if $q=1$, then both the two sides of \eqref{eq:QUE} turn out to vanish on all levels, for which we refer to Remarks \ref{vanish@1/2} and \ref{T=0}.
\end{mrema}

Theorem \ref{secondary} treats only Eisenstein series attached to the cusp $\infty$, but for arbitrary central characters.
The case where $\chi$ is primitive
has some simplifications that enable us to handle Eisenstein series attached to more general cusps.
\begin{mtheo}\label{thirdary}
Suppose $\chi \pmod{N}$ is primitive, and $\mathfrak{a}$ is a cusp singular for $\chi$.  Then \eqref{eq:QUE} holds for $E=E_{\mathfrak{a}}(z,s,\chi)$. 
\end{mtheo}

\begin{mrema}
Theorems \ref{secondary} and \ref{thirdary} also hold if $E=E_{\mathfrak{a}}(z,s,\chi)$ is of weight one. In this scenario, the error terms can be similarly bounded, while the main term are formally the same as \eqref{QUE}, for odd $\chi$. The main term resemblance in these two cases are well-reasoned, and we present more details in Remark \ref{weight1}. Because of the Maass raising (weight $+2$) and lowering (weight $-2$) operators \cite{Ma,DFI}, we then can study QUE for Eisenstein series of all integer weights. The difference between Eisenstein series in terms of weight is known explicitly by \cite{Y2}. See \cite[Theorem 1.6]{PRR2} for a related result.
\end{mrema}

\begin{mrema}
The term $4\Re \frac{L'}{L}(1+2iT,\overline{\psi})$ makes our QUE results qualitatively different from others that we have mentioned above. Since it is unknown if $4\Re \frac{L'}{L}(1+2iT,\overline{\psi}) = o_T(\log{q})$, we must include it as part of the main term.  Of course, such a bound holds on GRH (see \cite[Theorem 5.17]{IK}). This extra term in turn connects QUE for $T \approx 0$ and Siegel zeros. 

One can also surely adapt our techniques to treat some other cases, such as letting $\Re s \neq 1/2$, but we refrain from considering these generalizations in favor of simplicity in exposition.
\end{mrema}

\subsection{Shrinking sets in the level aspect}\label{section:shrinking}
In (\ref{KMV}), (\ref{K}), and Theorems \ref{secondary} and \ref{thirdary}, the test function $\phi$ is assumed to be $SL_2(\mathbb{Z})$-invariant.
A mild generalization of \eqref{KMV} is to fix a positive integer $M$ and a test function $\phi = \phi^{\scriptscriptstyle{(M)}}$ on $Y_0(M)$, and to confine $N \equiv 0 \pmod{M}$. See \cite{Hu} for a similar generalization to \cite{NPS}.
In analogy to the shrinking set version of QUE, where $\phi= \phi_t$ is allowed to change with the spectral parameter $t$, we are  led to consider the much more difficult generalization of letting $\phi$ depend on $N$.  A natural way to do this is to let $M$ grow  with $N$, constrained by $M|N$, and to choose $\phi=\phi^{\scriptscriptstyle{(M)}}$ on $Y_0(M)$ depending on $M$. To maintain uniform analytic properties of the test functions $\phi^{\scriptscriptstyle{(M)}}$ of varying levels, we often make the following system of choices.

%Berndt says mathematical environments should not be contained in a sentence, see paragraph 3, page 3 of https://faculty.math.illinois.edu/~berndt/writingmath.pdf. I tend to agree with him.

\begin{mconv}\label{shrinkage}
Once and for all fix an $SL_2(\mathbb{Z})$-invariant smooth function $\phi_{_0}=\phi^{\scriptscriptstyle{(1)}}$ with compact and connected support.  For simplicity, suppose that the support of $\phi_{_0}$, when restricted to the standard fundamental domain $\mathcal{D}$ of $SL_2(\mathbb{Z})$, is contained in its interior. %Under the canonical projection from $Y_0(M)$ to $Y_0(1)$, the support of $\phi_{0}$ pulls back to $\nu(M) = [\Gamma_0(1):\Gamma_0(M)]$ connected components in $Y_0(M)$.  
Suppose that $\Gamma_0(1) = \cup_{j=1}^{\scriptscriptstyle{\nu(M)}}  \gamma_j \Gamma_0(M) $ as a disjoint coset decomposition.
For each positive integer $M$, choose $\phi^{\scriptscriptstyle{(M)}} = \phi_j^{\scriptscriptstyle{(M)}}$ to be one of the following $\nu(M)$ functions.  Set $\phi_j^{\scriptscriptstyle{(M)}}(\gamma_k \Gamma_0(M)  z)$  equal to $\phi_{
_0}(z)$ if $j=k$, and zero if $j \neq k$, where $k\in \{ 1,2,..., \nu(M)\}$.  
One can interpret this definition intuitively by noting that $\cup_{j=1}^{\scriptscriptstyle{\nu(M)}} \gamma_j \mathcal{D}$ is a fundamental domain for $Y_0(M)$, and so $\phi_j^{\scriptscriptstyle{(M)}}$ agrees with $\phi_{_0}$ on one translate of $\mathcal{D}$ and vanishes at all others.
\end{mconv}
The system of test functions satisfying Convention \ref{shrinkage} has the following pleasant properties.  We have $\phi_{_0} = \sum_{j=1}^{\scriptscriptstyle{\nu(M)}} \phi_j^{\scriptscriptstyle{(M)}}$, where the supports of these $\phi_j^{\scriptscriptstyle{(M)}}$ are pairwise disjoint.  Moreover, we have $\int_{Y_0(M)} \phi_j^{\scriptscriptstyle{(M)}} d \mu = \int_{Y_0(1)} \phi_{_0} d\mu$, for each $j$.
%
% \begin{itemize}
%     \item for each $N$, pick arbitrary pullback $\phi^{\scriptscriptstyle{(M)}}$ of $\phi^{\scriptscriptstyle{(1)}}$ under the canonical projection from $\mathcal{F}_M$ to $\mathcal{D}$, so that Supp $\phi^{\scriptscriptstyle{(M)}}$ is a connected component of Supp $\phi_{_0}$ in $\mathcal{F}_M$. 
% \end{itemize}
%
%Then, what can be said about (\ref{K}) for $\phi = \phi^{\scriptscriptstyle{(M)}}$ defined above? 
Since
\begin{align*}
M^{-1-\varepsilon} \ll_{_{\phi_0}} \frac{\Vol(\Supp(\phi^{\scriptscriptstyle{(M)}})) }{\langle 1,1 \rangle_{_M}} \ll_{_{\phi_0}} M^{-1+\varepsilon},
\end{align*}
we intuitively see that $\Supp \phi^{\scriptscriptstyle{(M)}}$
``shrinks'', if $M \rightarrow \infty$ as $N \rightarrow \infty$.

\begin{mrema}
The above construction is merely one way of generating 
a system of test functions
that looks natural. 
Part of such an idea is borrowed from \cite{K}.
A similar treatment is adopted in \cite[Theorem 1.4]{KM} on counting Heegner points with changing levels. 
\end{mrema}

%The results from Section \ref{section:firstresults}  concern test functions of level $1$.  As discussed in Section \ref{section:shrinking}, for the level aspect version of shrinking sets we wish the test function to vary with the level.
\begin{mtheo}\label{main}
Let $E$ be as in Theorem \ref{secondary}.
Choose a system of test functions according to Convention \ref{shrinkage}.
Then there exists $\mathcal{E}\in \mathcal{A}(Y_0(N))$, such that $|E|^2-\mathcal{E}\in L^2(Y_0(N))$, and
\begin{equation}
\label{eq:mainthmErrorTerm}
 \langle |E|^2-\mathcal{E}, \phi \rangle_{_N} \ll_{_\varepsilon, _T, _{\phi_{_0}}} N^{-\frac{1}{2}+\varepsilon} (\tfrac{N}{q})^{\theta} Q(M,q) \norm{\phi_{_0}}_{_2},
\end{equation}
with 
\begin{equation*}
Q(M,q) =   M^{\frac{1}{4}} q^{\frac{3}{8}} + M^{\frac{1}{2}} (M,q)^{\frac{1}{4}} q^{\frac{1}{4}}.
\end{equation*}
Under the generalized Lindel\"of hypothesis, \eqref{eq:mainthmErrorTerm} holds with $Q(M,q) = \sqrt{M}$.
Finally, we have
\begin{equation}
\label{eq:mainthemMainTerm}
 \langle \mathcal{E}, \phi \rangle_{_N} = \frac{\langle 1,\phi_{_0} \rangle_{_1}}{\langle 1,1 \rangle_{_M}} \Big(\log \frac{N^2}{M(M,N/q)} + 4\Re \frac{L'}{L}(1+2iT,\overline{\psi})\Big) + O_{T,\phi_{_0}}\Big( \frac{( \log\log N )^5}{\langle 1,1 \rangle_{_M}} \Big) + \alpha_{\phi},
\end{equation}
where $\alpha_{\phi}$ is a quantity (see (\ref{alpha}) for an expression) satisfying 
\begin{equation}
\label{eq:alphaphibound}
|\alpha_{\phi}| \ll_{\phi_{_0}, _T} (\log\log 100 M)^3.
\end{equation}
\end{mtheo}
Note that if  $M \ll N^{\frac{1}{10} - \delta}$, then
the bound in \eqref{eq:mainthmErrorTerm} is 
better than the first displayed main term in \eqref{eq:mainthemMainTerm} of size $\approx M^{-1+o(1)} \log N$.  This is analogous to the power-saving error term in the QUE problem for Eisenstein series of level $1$ in the spectral aspect, as in \cite{Y1}.

\begin{mrema}
Theorem \ref{main} also holds for $E=E_{\mathfrak{a}}(z,\tfrac{1}{2}+iT,\chi)$ with $\chi$ (mod $N$) primitive as in Theorem \ref{thirdary}, and pleasantly, its corresponding $\mathcal{E}$ is much simpler (see Proposition \ref{id'}).
\end{mrema}
  
From the fact that $\langle 1, 1 \rangle_{_M} = M^{1+o(1)}$ 
, we may derive the following weak corollary.
\begin{mcoro}\label{sublog}
Under the assumptions of Theorem \ref{main},  QUE holds for all $M\ll (\log N)^{1-\delta}$ for $\delta >0$, and specifically when $M$ is a constant.
\end{mcoro}

\subsection{Main term discussion}\label{section:mainterm}
To our surprise, if we construct the system of test functions according to Convention \ref{shrinkage}, then QUE turns out not to hold for all test functions $\phi=\phi_j^{\scriptscriptstyle{(M)}}$, at least, if $M\gg N^{\delta}$ for some $\delta >0$.  The problem is that for some choices of $\phi$, the contribution of $\alpha_{\phi}$ to the main term is dominant and large enough to show that QUE does not hold.  In retrospect, one might expect problematic behavior for test functions with support escaping too quickly into a cusp.  This is clear in the level $1$ case (in the spectral aspect), since very high in the cusp the Eisenstein series is well-approximated by its constant term. In the level aspect, it is a bit tricky to say what it means for a test function to have support escaping into a cusp, not least because the cusp can be changing with the level.

To help explain the complication caused by $\alpha_{\phi}$, we study the case when $M\mid N$ is prime with $M \gg (\log N)^{1+\delta}$ and $\chi\thinspace (\text{mod } N)$ is primitive. We conclude here and leave the computation in Section \ref{non}. Let $G(z)$ denote the constant term in the Laurent expansion of $E(z,s)$ around $s=1$ (see \cite[(22.69)]{IK} for an expression), which is $SL_2(\mathbb{Z})$-invariant, and which satisfies $G(x+iy) \sim y$  for $y \rightarrow \infty$.
Then
\begin{align}\label{1.8.1}
    \langle \mathcal{E}, \phi \rangle_{_N} = c_0 \langle 1,\phi\rangle_{_M} + c_1 \langle G,\phi \rangle_{_M} + c_M \langle G|_{_M}, \phi \rangle_{_M},
\end{align}
where $G|_M(z)= G(Mz)$ is a $\Gamma_0(M)$-invariant function,
\begin{align}\label{1.8.2}
    c_0 = \frac{1}{\langle 1,1 \rangle_{_M}} \Big(\log \frac{N^2}{M} + 4\Re \frac{L'}{L}(1+2iT,\overline{\chi}) + O_{T, \phi_{_0}}( 1 ) \Big),
\end{align}
and $c_1,c_M =  M^{-1} + O(M^{-2})$. 
The term $c_0 \langle 1, \phi \rangle_{_M}$ is the naively-expected main term.  If $\phi = \phi_j^{\scriptscriptstyle{(M)}}$ is chosen according to Convention \ref{shrinkage}, then note $\langle G, \phi \rangle_{_M} = \langle G, \phi_{_0} \rangle_{_1}$, which is independent of $j$ and $M$, so the term $c_1 \langle G, \phi \rangle_{_M}$ is bounded acceptably. However, the term $c_M \langle G|_{_M}, \phi \rangle_{_M}$ may be much larger than the expected main term, as we now explain.  Suppose that the restriction of $\phi_{_0}$ to the standard fundamental domain $\mathcal{D}$ for $Y_0(1)$ has support with $2 \leq y \leq 3$ and that $\phi_{_0}$ is non-negative.  There exists a fundamental domain $\mathcal{F}_M$ for $Y_0(M)$ so that $\mathcal{D} \subset \mathcal{F}_M$, and there exists a value of $j$ so that $\phi_j^{\scriptscriptstyle{(M)}}(z) = \phi_{_0}(z)$ for $z \in \mathcal{D}$, and $\phi_j^{\scriptscriptstyle{(M)}}(z) = 0$ for $z \in \mathcal{F}_M$, $z \not \in \mathcal{D}$.  For this value of $j$, we have
\begin{equation*}
 c_M \langle G|_{_M}, \phi \rangle_{_M} \approx M^{-1}  \int_2^3 \int_{0}^{1} G(Mz) \phi(z) \frac{dx dy}{y^2},
\end{equation*}
which can be $\asymp 1$, since $G(Mz) \sim M y$ uniformly on the region of integration (see Proposition \ref{Laurent}).
Note that in this situation, $c_M \langle G|_{_M}, \phi \rangle_{_M}$ is much larger than $c_0 \langle 1, \phi \rangle_{_M} \lessapprox M^{-1} \log{N}$.
This choice of $\phi = \phi_j^{\scriptscriptstyle{(M)}}$ should be interpreted as having support high in the cusp $\infty$.  Nevertheless, we have the following theorem, with an elementary proof in Section \ref{agreement}.
\begin{mtheo}\label{portion}
There exists an absolute constant $\delta >0$, such that for all primes $M$, there are at least $\delta M$ test functions $\{ \phi^{\scriptscriptstyle{(M)}}_j \}_{j=1}^{\scriptscriptstyle{\nu(M)}}$ chosen according to Convention \ref{shrinkage} satisfying the QUE conjecture on shrinking sets. That is, for these $\phi = \phi_j^{\scriptscriptstyle{(M)}}$, we have
\begin{equation*}
|c_1| \cdot |\langle G, \phi \rangle_{_M}| + |c_M| \cdot |\langle G|_{_M}, \phi \rangle_{_M}| \ll M^{-1} \| \phi_{_0} \|_{_1},
\end{equation*}
while the term $c_0 \langle 1, \phi \rangle_{_M}$ is expected  to be approximately $\frac{6 \log{N}}{\pi M} \langle 1, \phi_{_0} \rangle_{_1}$.
\end{mtheo}

\begin{mrema}
Note that $\mathcal{E}$ is a linear combination of Eisenstein series attached to cusps on level $N$, which appears naturally in Zagier's regularization process that we elaborate later in Section \ref{SDIR}.
From the above discussions we can see the mass distribution of $|E|^2$ can be extremely uneven over supports of $\phi^{\scriptscriptstyle{(M)}}_j$ for different $j$. We conjecture that this $\delta$ can be improved to $1-\varepsilon$ for general $M$.  Also, we have an estimation of $\sum_{\phi} \alpha_{\phi}$ in \eqref{consistencycheck}. 
\end{mrema}

\subsection{Limitations to QUE}
Recall that the second author proved (\ref{LS}) for $\phi_t$ with shrinking support of radius $r \gg t^{-\delta}$ as $t\rightarrow \infty$ for some $\delta>0$. A natural question is how large can this $\delta$ be.
Humphries \cite{Hum2} showed that $\delta$ cannot exceed $1$, as (\ref{LS}) then fails for infinitely many $z$'s. On the other hand, he proved small scale QUE holds for almost all $z\in \mathbb{H}$ as long as $\delta <1$ (see \cite[Corollary 1.20]{Hum2} for the precise statement). 

In the level aspect, the discussion in Section \ref{section:mainterm} shows that QUE does not hold for all systems of test functions constructed according to Convention \ref{shrinkage}.  This is in contradiction to the claimed result of Koyama \cite[Theorem 1.3]{K}, which in our notation would correspond to $N=M$ prime and $q=1$. A recent corrigendum by Kaneko and Koyama \cite{KK} rewrites \cite[Theorem 1.3]{K} in the form of our (\ref{K}), and we refer the readers to it for more details.

\subsection{Strategy of the proof and QUE for newform Eisenstein series}
The reader may wonder why all of our QUE results are limited to only certain types of Eisenstein series.  It is a natural question to prove QUE for general newform Eisenstein series (see Section \ref{E} for definition), but unfortunately it does not appear that the inner products $\langle |E|^2, u_j \rangle_{_N}$ are computed in full detail in the literature.   
 This appears to be the only obstacle, as we expect that our techniques can be adapted to treat $\langle \mathcal{E}, \phi \rangle$ for the newform Eisenstein series.  Moreover, we remark that Theorem \ref{thirdary} does indeed treat all newform Eisenstein series of squarefree level or of primitive central character (see Remark \ref{squarefree} below for justification). 
 Paul Nelson has kindly informed us that the desired inner products may be computed using \cite[(4.26)]{MV1} and \cite[Theorem 49, part II]{N2}, but we leave this pursuit for a future occasion.
 
In broad strokes, the strategy for a proof of QUE (for cusp forms) is well-known.  Via a spectral decomposition and calculation of period integrals due to Watson/Ichino \cite{Wa, Ichino}, the problem reduces to a sufficiently strong subconvexity bound for certain triple product $L$-functions.  Unfortunately, power-saving subconvexity bounds in this generality have not been proved. 
A pleasant feature of the QUE problem for Eisenstein series is that the relevant $L$-functions factor into lower degree $L$-functions, for which subconvexity is known.

In practice, there are two main obstacles for proving Theorems \ref{secondary}, \ref{thirdary} and \ref{main}. The first difficulty is that $|E|^2$ is not in $L^2(Y_0(N))$, so the spectral decomposition can not be applied directly.   Our work-around for this problem is to execute a regularization procedure of 
Zagier \cite{Za} and Michel and Venkatesh \cite{MV1}.   We construct $\mathcal{E}$, a linear combination of Eisenstein series of level $N$ and trivial central character, so that $|E|^2 - \mathcal{E} \in L^2(Y_0(N))$.    The spectral decomposition can be applied to $|E|^2 - \mathcal{E}$, and the aforementioned subconvexity bounds eventually lead to a satisfactory estimate on this quantity.

The next significant problem is to asymptotically evaluate $\langle \mathcal{E}, \phi \rangle_{_N}$ as accurately as possible.   For this, we need to identify $\mathcal{E}$, which in turn requires
a careful study of the growth of $|E|^2$ at all the cusps, not just the ones that are singular with respect to the central character $\chi$.
This necessitates the precise calculation of the Fourier expansion of the Eisenstein series $E_{\mathfrak{a}}$. Koyama \cite{K} carried this out in the case that $N$ is prime. Recently, the second author \cite{Y2} developed explicit formulas for the Fourier expansions of a larger collection of Eisenstein series, including the case of $E_{\infty}^{\scriptscriptstyle{(N)}}$ for arbitrary $N$ and any central character, which is vital for the calculation of $\mathcal{E}$.
The function $\mathcal{E}$ is given in Proposition \ref{id} below.

\begin{mrema}\label{weight1}
When $E$ is  weight one, we can similarly obtain an $\mathcal{E}$ so that $\langle |E|^2-\mathcal{E}, \phi \rangle_{_N}$ has a power-saving bound as well. For the main term, $\mathcal{E}$ is again a linear combination of Eisenstein series of trivial central characters and weight zero. The only difference is $\mathcal{E}$ has different coefficients for each $E_{\mathfrak{a}}(z,s_{\mathfrak{a}})$, as the cuspidal behavior of $|E|^2$ depends on its weight. However, as one sees in the proof of Proposition \ref{id} or \ref{id'}, the coefficients are products of the entries of the scattering matrix, which does not change much under the weight shift. See \cite[Sec. 2]{Hum1} for the computation in the case of primitive $\chi$.
\end{mrema}

\subsection{Structure of the paper and sketch of proof of \eqref{eq:mainthmErrorTerm}}
To expose everything as clearly as possible, we initially prove Theorem \ref{main}, which contains Theorem \ref{secondary}. The main body of the proof lies in Sections \ref{spec}--\ref{MTE}, for which we sketch the argument for \eqref{eq:mainthmErrorTerm} later in this subsection; the supportive part consists of prerequisites about cusps in Section \ref{W}, Eisenstein series featured by a comprehensive description of their cuspidal behaviors in Section \ref{E0}, and regularized integrals in Section \ref{SDIR}. Finally, we prove Theorem \ref{thirdary} in Section \ref{othercusps}. 

The spectral decomposition applied to $\langle |E|^2 - \mathcal{E},\phi \rangle_{_N}$ gives
\begin{align*}
\langle |E|^2 - \mathcal{E},\phi \rangle_{_N} \approx \sum_{t_j \ll T} \sideset{}{^*}\sum_{u_j} \langle |E|^2, u_j \rangle_{_N}  \langle u_j, \phi \rangle_{_M} + \text{continuous spectrum},
\end{align*}
where the inner sum is over all $L^2(Y_0(M))$-normalized Hecke--Maass newforms of level $M$ with spectral parameter $t_j$, and recall that $E=E_{\infty}(z,\tfrac{1}{2}+iT,\chi)$. This regularized spectral decomposition is the topic of Section \ref{spec}, and Section \ref{ETE} mainly focuses on the following estimation.
\begin{mprop}\label{et}
With the above notations, we have
\begin{align*}
\langle |E|^2, u_j \rangle_{_N} \ll_{T, t_j} N^{-\frac{1}{2}+\varepsilon}M^{-\frac{1}{2}} (\tfrac{N}{q})^{\theta} |L(\tfrac{1}{2},u_j) L(\tfrac{1}{2}+2iT, u_j\otimes \overline{\psi})|.
\end{align*}
\end{mprop}

The following crucial subconvexity bound for twisted $L$-functions then finishes the job.
\begin{mtheo}[Blomer, Harcos \cite{BH}]\label{BH}
If $\psi$ is primitive $(\text{mod } {q})$ and $u_j$ is a newform of level $M$, then
\begin{align*}
    L(\tfrac{1}{2}+2iT, u_j\otimes \psi) \ll (|T|+1)^{\frac{1}{2}} ( M^{\frac{1}{4}}  q^{\frac{3}{8}} +  M^{\frac{1}{2}} (M,q)^{\frac{1}{4}} q^{\frac{1}{4}}).
\end{align*}
\end{mtheo}
The contribution of the continuous spectrum to $\langle |E|^2 - \mathcal{E}, \phi\rangle_{_N}$ is similar. Section \ref{MTE} addresses the main terms, about which we have briefly discussed earlier in this section.

\subsection{Acknowledgements}
We thank Junehyuk Jung, Shin-ya Koyama, Riad Masri and Peter Sarnak for their discussions on this material. We also wish to express our gratefulness to Peter Humphries and Paul Nelson for their insightful comments on multiple places of this paper. In addition, we hope to write out our indebtedness to Ikuya Kaneko, who proofread this paper and helped us improve its readability by a great margin.  Finally, we appreciate the referees' careful reading and suggestions for improvements.

\section{Cusps and their widths}\label{W}
 It is well-known that $\Gamma_0(N)= \{ (\begin{smallmatrix} a&b\\c&d \end{smallmatrix}) \in SL_2(\mathbb Z) | \medspace c\equiv 0 \pmod N \}$ acts on $\mathbb H$ via $ (\begin{smallmatrix}a&b\\c&d\end{smallmatrix})z \mapsto \frac{az+b}{cz+d}$.
In this section we introduce some background knowledge of cusps on $\Gamma_0(N)$. We counsel experienced readers to skip this section except for Section \ref{RW} on relative width, and refer other readers to \cite[Section 3.4]{NPS} and \cite[Sections 2.1--2.4]{I1} for more details.

\subsection{Cusps} 
The group action can be extended to $\mathbb P^1(\mathbb Q)$, the set of \textit{cusps}. We often employ the letters $\mathfrak{a,b,c},$..., to denote cusps. We say two cusps $\mathfrak{a}$ and $\mathfrak{b}$ are equivalent on level $N$ and write $\mathfrak{a}\overset{\scriptscriptstyle N}{=}\mathfrak{b}$, if there exists $\gamma \in \Gamma_0(N)$ such that $\mathfrak{a}=\gamma\mathfrak{b}$. That is to say, equivalence classes of cusps on level $N$ are the $\Gamma_0(N)$-orbits in $\mathbb P^1(\mathbb Q)$.

By \cite[Proposition 2.6]{I1}, a full set of inequivalent cusps on level $N$ can be written as
\begin{align}\label{DS}
\begin{split}
    \mathcal{C}(N)&:=\{ \mathfrak{a} \big| \medspace \mathfrak{a}=\tfrac{u}{f}, f\mid N, u=\min \mathcal{R}(N,f,v), v\in \big( \mathbb{Z}/N\mathbb{Z} \big)^{\times}  \}, \text{ with}\\
\mathcal{R}(N,f,v)&:= \{u\equiv v \mymod{(f,N/f)}, {u \geq 1} \}.
\end{split}
\end{align}

\begin{mrema}
Throughout this paper we write $u_{\mathfrak{a}}$ and $f_{\mathfrak{a}}$ such that $\mathfrak{a}\overset{\scriptscriptstyle N}{=}\tfrac{u_{\mathfrak{a}}}{f_{\mathfrak{a}}}\in \mathcal{C}(N)$, if necessary.
Also, if we write $\frac{u}{f}\in \mathcal{C}(N)$, then we always assume that the fraction is in the lowest terms.
\end{mrema}

Let $\Gamma_{\mathfrak{a}}^N$ be the stabilizer of $\mathfrak{a}$ in $\Gamma_0(N)$. It is clear that for all $N$,
$
    \Gamma_{\infty}^N = \{ \pm (\begin{smallmatrix} 1 & n \\ 0 & 1 \end{smallmatrix}) | \medspace  n \in \mathbb{Z} \}$,
so we may write $\Gamma_{\infty}$ as well. In addition, there are \textit{scaling matrices} $\sigma_{\mathfrak{a},\scriptscriptstyle{N}} \in SL_2(\mathbb R)$ such that $\sigma_{\mathfrak{a},\scriptscriptstyle{N}} \infty = \mathfrak{a}$, and $\sigma_{\mathfrak{a},\scriptscriptstyle{N}}^{-1} \Gamma_{\mathfrak{a}}^N \sigma_{\mathfrak{a},\scriptscriptstyle{N}} = \Gamma_{\infty}$. If the level is clear, we may suppress $N$ in these symbols.

\subsection{(Absolute) width}\label{AW}
If $\tau\in \Gamma=SL_2(\mathbb Z)$ and $\tau \infty = \mathfrak{a}$, then $\tau^{-1}\Gamma_{\mathfrak{a}}^N\tau$ is a subgroup of $\Gamma_{\infty}$. Since $\tau \Gamma_{\infty} \tau^{-1} = \Gamma_{\mathfrak{a}}^1$, we have $[\Gamma_{\infty}:\tau^{-1}\Gamma_{\mathfrak{a}}^N\tau]=[\Gamma_{\mathfrak{a}}^1:\Gamma_{\mathfrak{a}}^N]$, which does not depend on the choice of $\tau$. Define this index as the (\textit{absolute}) \textit{width} of $\mathfrak{a}$ on level $N$ and write it $W^1_N(\mathfrak{a})$.

\begin{mconv}
When there is no ambiguity on levels, we may write the (absolute) width of $\mathfrak{a}$ by $W_{\mathfrak{a}}$ as well. Width of a cusp is a common terminology, so we add ``absolute'' only if it is necessary to distinguish it from relative width introduced in the following subsection.
\end{mconv}

\begin{mrema}\label{gamma2sigma}
For future usage we cite \cite[(2.31)]{I1} to note that for fixed $\gamma_{\mathfrak{a}} \in SL_2(\mathbb{Z})$ sending $\infty$ to $\mathfrak{a}$,  $
    \gamma_{\mathfrak{a}} \Big(\begin{smallmatrix} W_{\mathfrak{a}}^{1/2} &0 \\ 0& W_{\mathfrak{a}}^{-1/2} \end{smallmatrix} \Big) $ serves as a scaling matrix $\sigma_{\mathfrak{a}}=\sigma_{\mathfrak{a},N}$.
\end{mrema}

\begin{mlemma}\cite[(2.29)]{I1}\label{width}
For each $\mathfrak{a} = \frac{u}{f}\in\mathcal{C}(N)$ in (\ref{DS}), we have
\begin{align*}
    W_{\mathfrak{a}} = \frac{N}{(N, f^2)}.
\end{align*}
\end{mlemma}

\begin{mrema}\label{ky}
Let $M\mid N$, and $\mathfrak{a}=\tfrac{u}{f}\in\mathcal{C}(N)$.
Then by \cite[Proposition 3.1]{KY}, $\mathfrak{a}$ is equivalent to a cusp of the form $\tfrac{u'}{(M,f)} \in \mathcal{C}(M)$, with width
\begin{align*}
    W^1_M(\mathfrak{a})=\frac{M}{(M, (M,f)^2)}.
\end{align*}
\end{mrema}

\subsection{Relative width}\label{RW}
Now we fix $\Gamma_0(N)$ but let $\Gamma = \Gamma_0(M)$ for any $M\mid N$ instead. We define the index $[\Gamma_{\mathfrak{a}}^M : \Gamma_{\mathfrak{a}}^N]$ as the \textit{relative width} of $\mathfrak{a} \in \mathcal{C}(N)$ \textit{from level} $M$, and denote it by $W^M_N(\mathfrak{a})$. Note that the absolute width is a special case of the relative width when $M=1$. 

\begin{mrema}\label{widtheq}
From the definition we can also see if $\mathfrak{a}\overset{\scriptscriptstyle N}{=} \mathfrak{b}$, then $W^M_N(\mathfrak{a})=W^M_N(\mathfrak{b})$. This results from the fact $\Gamma_{\mathfrak{b}}^*= \tau \Gamma_{\mathfrak{a}}^* \tau^{-1}$, for any $\tau \in \Gamma_0(N)$ with $\tau \mathfrak{a} = \mathfrak{b}$ and $*=M,N$. %not sure I like my proposed change in notation...]
\end{mrema}

The following lemma follows directly from the definition.
\begin{mlemma}\label{relative}
For each cusp $\mathfrak{a} \in \mathcal{C}(N)$, we have
\begin{align*}
W^M_N(\mathfrak{a}) = \frac{W^1_N(\mathfrak{a})}{W^1_M(\mathfrak{a})}.
\end{align*}
\end{mlemma}

\begin{mlemma}\label{altwidth}
For $\mathfrak{a}, \mathfrak{b} \in \mathcal{C}(N)$, we have
\begin{align*}
\#\{ \gamma \in \Gamma_0(N)\backslash \Gamma | \medspace \gamma \mathfrak{b}
\overset{\scriptscriptstyle N}{=}\mathfrak{a} \} = \begin{cases} W^M_N(\mathfrak{a}) & \text{ if } \mathfrak{a}\overset{\scriptscriptstyle M}{=} \mathfrak{b}; \\ 0 & \text{ otherwise}. \end{cases}
\end{align*}
\end{mlemma}

\begin{proof}
If $\mathfrak{a}$ is not $\Gamma$-equivalent with $\mathfrak{b}$, then  the set is empty. 
Now assume $\mathfrak{a}\overset{\scriptscriptstyle M}{=} \mathfrak{b}$ with  $\tau\mathfrak{a} = \mathfrak{b}$ for some $\tau\in \Gamma$. We have the following bijective map
\begin{align*}
\{ \gamma \in \Gamma_0(N)\backslash \Gamma | \medspace \gamma\mathfrak{b}\overset{\scriptscriptstyle N}{=}\mathfrak{a} \} &\rightarrow \{ \gamma \in \Gamma_0(N)\backslash \Gamma | \medspace \gamma\mathfrak{a}\overset{\scriptscriptstyle N}{=}\mathfrak{a} \} \\
\gamma &\mapsto \gamma\tau
\end{align*}
so it suffices to compute $\# S_{\mathfrak{a}}$, where $S_{\mathfrak{a}}=\{ \gamma \in \Gamma_0(N)\backslash \Gamma | \medspace \gamma\mathfrak{a}\overset{\scriptscriptstyle N}{=}\mathfrak{a} \}$. Note that $\Gamma_{\mathfrak{a}}^M$ acts transitively on $S_{\mathfrak{a}}$ (on the right) with stabilizer $\Gamma_{\mathfrak{a}}^N$.  Hence, by the Orbit-Stabilizer Theorem (see e.g., \cite[Chapter 5, Proposition (7.2)]{Ar}), we have $\#S_{\mathfrak{a}} = [\Gamma_{\mathfrak{a}}^M: \Gamma_{\mathfrak{a}}^N] = W^M_N(\mathfrak{a})$.
\end{proof}

\subsection{Singularity}\label{singularity}
Given an even Dirichlet character $\chi \pmod{N}$, i.e., $\chi(-1)=1$, we define 
\begin{equation*}
    \chi:  \Gamma_0(N) \rightarrow \mathbb{C}^*
\end{equation*}
by $\chi(\gamma) = \chi(d_{\gamma})$, where $d_{\gamma}$ stands for the lower-right entry of $\gamma$. It is easy to see that $\chi$ preserves multiplication of the two sides, and hence it is a group homomorphism.

\begin{mconv}\label{simeq}
We write $\chi_1 \simeq \chi_2$ if they are induced by the same primitive character.
\end{mconv}

We say $\mathfrak{a}$ is \textit{singular} for $\chi$, if the kernel of $\chi$ contains $\Gamma_{\mathfrak{a}}^{{N}}$. 
If $\chi_1\simeq \chi_2$, then the singularity of $\mathfrak{a}$ for $\chi_1$ is equivalent to that for $\chi_2$.
For fixed $\chi \pmod{N}$, singularity and non-singularity of a cusp extends to its $\Gamma_0(N)$-equivalence class, for the same reason as for Remark \ref{widtheq}.
\begin{mconv}\label{cchiN}
For $\chi \pmod{N}$, we write the subset  of singular cusps for $\chi$ by $\mathcal{C}_{\chi}(N)$. Note $\mathcal{C}_{\chi}(N)=\mathcal{C}(N)$ if $\chi$ is trivial.
\end{mconv}

We have a criterion for singularity from \cite[Lemma 5.4]{Y2}. Recall from Convention \ref{conv1} that $q$ is the conductor of $\chi$.
\begin{mprop}\label{criterion}
The cusp $\tfrac{u}{f}\in \mathcal{C}(N)$ is singular for $\chi$ if and only if $q \mid [f, \frac{N}{f}]$.
\end{mprop}

One interesting case is when $\chi$ is primitive $\negthinspace \negthinspace \pmod{N}$. By Proposition \ref{criterion}, only cusps $\mathfrak{a}=\frac{u}{f}\in \mathcal{C}(N)$ with $(f, N/f)=1$ are singular for $\chi$. Moreover, from (\ref{DS}) we can see $u=1$. These cusps are known as the \textit{Atkin--Lehner cusps}.

\section{Eisenstein series of weight zero}\label{E0}
This section deals with knowledge about Eisenstein series of weight zero. We suggest advanced readers skip this section with a glance on Propositions \ref{scatteringmatrix} and \ref{nonsingular} on descriptions of their cuspidal behaviors. Good references include \cite[Chapter 4]{DS} and \cite{I1}.

\subsection{Two kinds of Eisenstein series}\label{E}
On level $N$, there are Eisenstein series \textit{attached to cusps} and Eisenstein series \textit{attached to characters}.

%\begin{itemize}
%\item
The Eisenstein series of central character $\chi \pmod{N}$ attached to the cusp $\mathfrak{a}$ is
\begin{align*}
%\label{Ea}
E_{\mathfrak{a}}(z,s, \chi)= \sum_{\gamma\in\Gamma_{\mathfrak{a}}\backslash\Gamma_0(N)}\overline{\chi}(\gamma)(\text{Im}\thinspace \sigma_{\mathfrak{a}}^{-1}\gamma z)^s.
\end{align*}
To make this well-defined, we require $\chi$ to be even, and $\mathfrak{a}$ to be singular for $\chi$. The definition does not depend on the choice of $\sigma_{\mathfrak{a}}$. Since $E_{\gamma \mathfrak{a}} = \overline{\chi}(\gamma)E_{\mathfrak{a}}$ for $\gamma\in\Gamma_0(N)$, we can always represent $E_{\mathfrak{a}}$ in terms of $E_{\mathfrak{a}'}$ with $\mathfrak{a}'\in \mathcal{C}_{\chi}(N)$ (see Convention \ref{cchiN} for definition and Remark \ref{cb} for practice). 
%\end{itemize}
% \begin{itemize}
 %\item

 For Dirichlet characters $\chi_i \pmod{q_i}$ with $i=1,2$, having the same parity, the Eisenstein series attached to $\chi_1, \chi_2$ is
\begin{align*}
%\label{Echichi}
E_{\chi_1, \chi_2}(z,s)=\frac{1}{2}\sum_{(c,d)=1}\frac{(q_2 y)^s \chi_1(c)\chi_2(d)}{|c q_2 z +d|^{2s}}.
\end{align*}
If both $\chi_1$ and $\chi_2$ are primitive, $E_{\chi_1,\chi_2}$ is a \textit{newform} Eisenstein series of level $q_1q_2$.
%\end{itemize}

Both types of Eisenstein series converge absolutely for $\Re s >1$, with meromorphic continuations to $\mathbb{C}$.

\begin{mconv}\label{hidechi}
When $\chi=\chi_{_0,_N}$, we write $E_{\mathfrak{a}}(z,s)$ in short of $E_{\mathfrak{a}}(z,s, \chi)$.
If $N=1$, then the classical Eisenstein series $E$ is the only one in both types, so we write it in place of $E_{1,1}$.
If we want to emphasize $E_{\mathfrak{a}}$ is an Eisenstein series of level $N$, then we may write $E_{\mathfrak{a}}^{\scriptscriptstyle{(N)}}$ instead.
\end{mconv}

These two kinds of Eisenstein series are closely connected.  Recently, the second author \cite{Y2} found the change-of-basis formulas between them, which is also done by Booker, Lee, and Str\"ombergsson \cite{BLS}. %in the meantime. 
%Here we quote one of the two, that is helpful here.

\begin{mtheo}\label{c2c}\cite[Theorem 6.1]{Y2}
Keeping notations in Conventions \ref{conv1} and \ref{cchiN}, and denoting the Euler totient function by $\varphi$, we have for $\mathfrak{a}=\frac{u}{f} \in \mathcal{C}_{\chi}(N)$
\begin{multline*}
E_{\mathfrak{a}}(z,s, \chi)= \frac{W_{\mathfrak{a}}^{-s} f^{-s}}{\varphi((f,\frac{N}{f}))} \sum_{q_1\mid \frac{N}{f}} \sum_{q_2 \mid f} \sideset{}{^*}\sum_{\chi_1, \chi_2}\overline{\chi_2}(-u)\frac{L(2s, \chi_1\chi_2)}{L(2s,\chi_1\chi_2\chi_{_0,_N})} \\
\sum_{a\mid f} \sum_{b \mid \frac{N}{f}} \frac{\mu(a)\mu(b)\chi_1(b)\chi_2(a)}{(ab)^s}E_{\chi_1, \chi_2} \Big(\frac{bf}{a q_2}z, s \Big),
\end{multline*}
where the asterisked sum is over all primitive $\chi_i \pmod{q_i}, i=1,2$, satisfying $\chi_1 \overline{\chi_2} \simeq \chi$.
\end{mtheo}

\begin{mrema}\label{cb}
In \cite{Y2}, the cusp choice $\mathfrak{a}=\frac{1}{uf}$ was made, and we transfer it for convenience. It is remarked in \cite[Section 5.2]{Y2}, that for all $\frac{u}{f}\in \mathcal{C}(N)$, there is $\gamma\in \Gamma_0(N)$ such that $\gamma \frac{u}{f}=\frac{1}{uf}$, and has lower-right entry equal to $u \pmod{N}$. Then we have
\begin{align*}
    E_{\frac{u}{f}} = \chi(u) E_{\frac{1}{uf}}.
\end{align*}
\end{mrema}

We are interested in two special cases: when $f=N$, and when $q=N$.

Since $\infty \overset{\scriptscriptstyle N}{=} \frac{1}{N}$ via $\gamma=\big(\begin{smallmatrix} 1&0\\N&1 \end{smallmatrix}\big)$, we have $E_{\infty}=E_{\frac{1}{N}}$. By Theorem \ref{c2c}, we have
\begin{align}\label{expansion}
    E_{\infty}(z,s,\chi)= N^{-s} \frac{L(2s,\overline{\psi})}{L(2s, \overline{\chi})} \sum_{a \mid N} \frac{\mu(a)\overline{\psi}(a)}{a^s} E_{1,\overline{\psi}} \Big(\frac{N}{a q}z,s \Big).
\end{align}

If $\chi$ is primitive $\medspace\medspace\medspace\shortmod{N}$, then only Atkin--Lehner cusps are singular for it, as is discussed in Section \ref{singularity}. Assuming $\mathfrak{a}=\frac{1}{f}\in \mathcal{C}_{\chi}(N)$, we have
\begin{align}\label{expansionq=N}
    E_{\mathfrak{a}}(z,s,\chi) = N^{-s} E_{\chi_1, \chi_2}(z,s),
\end{align}
where $\chi_1$ is primitive $\mymod{N/f}$ and $\chi_2$ is primitive $\mymod{f}$, with $\chi = \chi_1 \overline{\chi_2}$.

\begin{mrema}\label{squarefree}
Now we see why Theorem \ref{thirdary} implies QUE for all newform Eisenstein series of squarefree levels. If $N$ is squarefree, then by definition, a newform Eisenstein series of level $N$ is $E_{\chi_1, \chi_2}(z,s)$ for some primitive $\chi_i$ mod $q_i$, $i=1,2$, with $q_1q_2=N$ and $(q_1, q_2)=1$. Then \eqref{expansionq=N} says $E=N^{s} E_{\frac{1}{q_2}}(z,s,\chi_1 \overline{\chi_2})$, to which Theorem \ref{thirdary} applies. 

In addition, if we relax the squarefree assumption on $N$ and instead assume $E=E_{\chi_1, \chi_2}$ is a newform Eisenstein series of level $N$ and primitive central character $\chi \simeq \chi_1 \overline{\chi_2} \pmod{N}$, for $\chi_i$ mod $q_i$, $i=1,2$, then since $q_1 q_2 = N$, we must have $(q_1,q_2)=1$. The above argument again shows QUE for $E=N^{s} E_{\frac{1}{q_2}}(z,s,\chi_1 \overline{\chi_2})$.
\end{mrema}

\subsection{Fourier expansions}
One merit of Eisenstein series attached to primitive characters is their explicit Fourier expansions with multiplicative Fourier coefficients.  Define the \textit{completed} Eisenstein series by
\begin{align*}
E^* _{\chi_1, \chi_2}(z,s) := \theta_{\chi_1,\chi_2}(s) E_{\chi_1, \chi_2}(z,s),
\end{align*}
with $\chi_i$ primitive $\mymod{q_i}$, $i=1,2$, and
\begin{align}\label{theta}
\theta_{\chi_1,\chi_2}(s) = \frac{q_2^s \pi^{-s}}{\tau(\chi_2)} \Gamma(s) L(2s, \chi_1 \chi_2). %= \frac{q_1^{-s}}{\tau(\chi_2)} \Lambda(2s,\chi_1\chi_2).
\end{align}
Then we have the Fourier expansion
\begin{align}\label{Fourier}
E^* _{\chi_1, \chi_2}(z,s) = e_{\chi_1, \chi_2}^*(y,s) + 2\sqrt{y} \sum_{n \neq 0} \lambda_{\chi_1, \chi_2}(n,s) e(nx) K_{s-\frac{1}{2}}(2\pi |n| y),
\end{align}
where the constant term is
\begin{align*}
e^*_{\chi_1, \chi_2}(y,s)= \delta_{q_1 =1} \theta_{1,\chi_2}(s) (q_2y)^s + \delta_{q_2 =1} \theta_{1,\overline{\chi_1}}(1-s) (q_1y)^{1-s},
\end{align*}
$\lambda_{\chi_1, \chi_2}(n,s) = \chi_2(\frac{n}{|n|}) \sum_{ab=|n|}\chi_1(a)\overline{\chi_2}(b) (\frac{b}{a})^{s-\frac{1}{2}}$, $\tau(\chi)$ is the Gauss sum of $\chi$, and $K_{\alpha}$ is the $K$-Bessel function of order $\alpha \in \mathbb{C}$, so that the series in (\ref{Fourier}) decays exponentially, as $y \rightarrow \infty$. 
See Huxley \cite{Hux}, and Knightly and Li \cite[Section 5.6]{KL} for more details.
\begin{mrema}\label{lambda1}
From the definition we see that when $s=\frac{1}{2}+iT$, $|\lambda_{\chi_1,\chi_2}(n,s)| \leq d(n) \ll n^{\varepsilon}$.
\end{mrema}
\begin{mrema}\label{regularity@1}
If $\chi$ is primitive $\mymod{q}$ for $q>1$, then $E_{\chi,\chi}(z,s)$ is regular at $s=1$.
\end{mrema}
\begin{mrema}\label{Hecke}
The newform Eisenstein series are eigenfunctions of all the Hecke operators $T_n$, and indeed $T_n E_{\chi_1, \chi_2}(z,s) = \lambda_{\chi_1, \chi_2}(n,s) E_{\chi_1, \chi_2}(z,s)$.
\end{mrema}

\begin{mrema}\label{vanish@1/2}
We can also see from \eqref{Fourier} that $E^*_{1,1}(z,s)$ is analytic for $s\in \mathbb{C}$ except $s=0,1$, and in particular, it is well-defined at $s=\tfrac{1}{2}$. On the other hand, since $\theta_{1,1}$ has a pole at $s=\tfrac{1}{2}$, then $E_{1,1}(z,\tfrac{1}{2})=0$.
Thus by \eqref{expansion}, we have $E_{\infty}(z,\tfrac{1}{2})=0$.
\end{mrema}

For future application, we write out two special cases. When $\chi_1=1$, and $\chi_2 = \overline{\psi}$ primitive $\mymod{q}$, we have
\begin{align}\label{specialf=N}
    E_{1,\overline{\psi}}(z,\tfrac{1}{2}+iT)= e_{1,\overline{\psi}}(y,\tfrac{1}{2}+iT) + 2\rho_{1,\overline{\psi}}(\tfrac{1}{2}+iT)\sqrt{y} \sum_{n\neq 0}  \lambda_{1,\overline{\psi}}(n) e(nx) K_{iT}(2\pi |n| y) ,
\end{align}
where $e_{\chi_1,\chi_2}(s) = \rho_{\chi_1, \chi_2}(s)e^*_{\chi_1, \chi_2}(y,s)$, $\rho_{\chi_1, \chi_2}(s)=\frac{1}{\theta_{\chi_1,\chi_2}(s)}$, $\lambda_{\chi_1, \chi_2}(n)=\lambda_{\chi_1, \chi_2}(n,\tfrac{1}{2}+iT)$, and 
\begin{align}\label{rho1}
\rho_{1,\overline{\psi}}(\tfrac{1}{2}+iT) = O(q^{\varepsilon} (1+|T|)^{\varepsilon} e^{\frac{\pi |T|}{2}})
\end{align}
by Stirling's formula, see e.g. \cite[(5.73)]{IK} and \cite[(11.18)]{MV2}.
Another case is when $q_1 q_2=N$ with $(q_1,q_2)=1$,  and $\chi_i$ is primitive $\mymod{q_i}$ for $i=1,2$. We then have
\begin{align}\label{specialq=N}
    E_{\chi_1,\chi_2}(z,\tfrac{1}{2}+iT) = \rho_{\chi_1,\chi_2}(\tfrac{1}{2}+iT) \sqrt{y} \sum_{n\neq 0} \lambda_{\chi_1,\chi_2}(n) e(nx) K_{iT}(2\pi |n| y),
\end{align}
and similarly,
\begin{align}\label{rho3}
\rho_{\chi_1,\chi_2}(\tfrac{1}{2}+iT) = O(N^{\varepsilon} (1+|T|)^{\varepsilon} e^{\frac{\pi |T|}{2}}).
\end{align}

Next we discuss some aspects of the Fourier expansion of $E_{\mathfrak{a}}(z,s, \chi)$.
For the following discussion, assume $\mathfrak{a,b}$ are cusps singular for $\chi$.
When $y\rightarrow\infty$ (see e.g., \cite[(13.15)]{I1})
\begin{align}\label{ab}
    E_{\mathfrak{a}}(\sigma_{\mathfrak{b}}z,s,\chi) = \delta_{\mathfrak{a}\mathfrak{b}}y^s + \varphi_{\mathfrak{ab}}(s,\chi) y^{1-s} + O(y^{-P}),
\end{align}
for all $P\in \mathbb{N}$, where $\delta_{\mathfrak{ab}} =1$ if $\mathfrak{a}\overset{\scriptscriptstyle{N}}{=}\mathfrak{b}$, and vanishes otherwise, and $\varphi_{\mathfrak{ab}}$ is meromorphic in $s\in \mathbb{C}$. Each $\varphi_{\mathfrak{ab}}$ is an entry of the scattering matrix, which we compute explicitly later. Iwaniec writes $\varphi_{\mathfrak{ab}}$ as an infinite sum, see \cite[(13.16)--(13.18)]{I1}, and we have an alternative finite expression in Proposition \ref{scatteringmatrix} below.
\begin{mconv}
Analogously to Convention \ref{hidechi}, if $\chi=\chi_{_0,_N}$, then we suppress it from $\varphi_{\mathfrak{ab}}(s,\chi)$; if necessary, we write $\varphi_{\mathfrak{ab}}^{\scriptscriptstyle{(N)}}$ to emphasize it comes from $E_{\mathfrak{a}}^{\scriptscriptstyle{(N)}}$.
\end{mconv}
\begin{mprop}[Selberg \cite{I1} (13.30)] \label{Selberg}
For $\Re s = \frac{1}{2}$, the matrix $\Phi(s,\chi)=\begin{pmatrix} \varphi_{\mathfrak{a}\mathfrak{b}}(s,\chi) \end{pmatrix}_{\mathfrak{a},\mathfrak{b}}$ is unitary. In particular, we have $\sum_{\mathfrak{a}\in \mathcal{C}_{\chi}(N)}|\varphi_{\infty\mathfrak{a}}(s,\chi)|^2 =1$ for $s=\frac{1}{2}+iT$.
\end{mprop}

\subsection{Functional equations}
Eisenstein series attached to Dirichlet characters satisfy the following simple functional equation. Recall $\sigma_{\mathfrak{a}}=\sigma_{\mathfrak{a},_N}$ is a scaling matrix as in Remark \ref{gamma2sigma}.
\begin{mprop}[Huxley \cite{Hux}]\label{FE}
For primitive $\chi_1$ and $\chi_2$, we have
\begin{align*}
    E^*_{\chi_1,\chi_2}(z,s)=E^*_{\overline{\chi_2},\overline{\chi_1}}(z,1-s).
\end{align*}
\end{mprop}

When $(q_1,q_2)=1$ and $\mathfrak{a}=\frac{1}{q_2}$, Weisinger \cite{We} essentially showed (see also \cite[(9.1)]{Y2})
\begin{align}\label{YW}
E_{\chi_1,\chi_2}|_{\sigma_{\mathfrak{a}}} = \epsilon_{\chi_1, \chi_2} E_{1,\chi_1 \chi_2},
\qquad
\text{where}
\qquad
|\epsilon_{\chi_1, \chi_2}| = 1.
\end{align}

\iffalse
\begin{mrema}
Together with (\ref{expansionq=N}), Proposition \ref{FE} says
\begin{align*}
\theta_{\chi_1,\overline{\chi_2}}(s)E_{\mathfrak{a}}(z,s,\chi) = \theta_{\chi_2,\overline{\chi_1}}(1-s) E_{\mathfrak{a}^*}(z,1-s,\overline{\chi}),
\end{align*}
where $\chi=\chi_1\chi_2$ is primitive $\mymod{N}$ with $\chi_1$ primitive $\mymod{N/f}$ and $\chi_2$ primitive $\mymod{f}$, $\mathfrak{a}=\frac{1}{f}$ is Atkin--Lehner, and $\mathfrak{a}^*=\frac{1}{N/f}$. In particular, when $s = \frac{1}{2} +iT$, we have
\begin{align*}
|E_{\mathfrak{a}}(z,\tfrac{1}{2}+iT,\chi)|^2 = |E_{\mathfrak{a}^*}(z,\tfrac{1}{2}-iT,\overline{\chi})|^2.
\end{align*} 
\end{mrema}
\fi

\subsection{Identifying traced Eisenstein series}
Define the \emph{trace operator} 
$\Tr^N_M : \mathcal{A}(Y_0(N)) \rightarrow \mathcal{A}(Y_0(M))$ via
\begin{align}\label{trace}
    f &\mapsto \sum_{\gamma\in \Gamma_0(N)\backslash\Gamma_0(M)}f|_{\gamma}.
\end{align}
Now we can determine the exact shape of $\Tr^N_M E^{\scriptscriptstyle{(N)}}_{\mathfrak{a}}(z,s)$ by \eqref{ab}.
\begin{mlemma}\label{relativetrace}
We have the following equality of meromorphic functions:
\begin{align*}
    \Tr^N_M E^{\scriptscriptstyle{(N)}}_{\mathfrak{a}}(z,s) = (W^M_N({\mathfrak{a}}))^{1-s} E^{\scriptscriptstyle{(
    M)}}_{\mathfrak{a}}(z,s).
\end{align*}
\end{mlemma}
\begin{mrema}
We have to point out that when $\mathfrak{a}$ is a cusp for $Y_0(N)$, there might be ambiguities for the symbol of $E_{\mathfrak{a}}^{\scriptscriptstyle{(M)}}$. However, since the central character is trivial, the choice of representative for $\mathfrak{a}$ in $Y_0(M)$ does not affect the resulted function, as mentioned in Section \ref{E}.
\end{mrema}
\begin{proof}
Let $\Re s >1$. By \cite[Lemma 6.4]{I2}, $\Tr^N_M E^{\scriptscriptstyle{(N)}}_{\mathfrak{a}}(z,s)$ is a linear combination of $E^{\scriptscriptstyle{(
    M)}}_{\mathfrak{b}}(z,s)$ for $\mathfrak{b} \in \mathcal{C}(M)$. Furthermore, this linear combination is unique, since the only linear combination of Eisenstein series that decays rapidly at all cusps is the zero combination. In light of this, if $\Tr^N_M E^{\scriptscriptstyle{(N)}}_{\mathfrak{a}}(\sigma_{\mathfrak{b},\scriptscriptstyle{M}} z, s) = c_{\mathfrak{b}}y^s + O(1)$ as $y\rightarrow \infty$, then it is identical to $\sum_{\mathfrak{b}}c_b E_{\mathfrak{b}}^{\scriptscriptstyle{(M)}}(z,s)$, because their difference  has rapid decay. Now we compute the $y^s$-coefficients.

For each $\mathfrak{b}$ pick $\sigma_{\mathfrak{b},\scriptscriptstyle{M}} = \gamma_{\mathfrak{b}} \big(\begin{smallmatrix} W^{1/2} & 0 \\ 0 & W^{-1/2} \end{smallmatrix} \big)$ as in Remark \ref{gamma2sigma}, where $\gamma_{\mathfrak{b}}\in SL_2(\mathbb Z)$, $\gamma_{\mathfrak{b}}\infty = \mathfrak{b}$, and $W=W^1_M(\mathfrak{b})$. As $y\rightarrow\infty$, we have by (\ref{ab}), Lemmas \ref{altwidth}, \ref{relative} and Remark \ref{widtheq},
\begin{align*}
    \Tr^N_M E^{\scriptscriptstyle{(N)}}_{\mathfrak{a}}(\sigma_{\mathfrak{b},\scriptscriptstyle{M}} z, s) &= \sum_{\gamma\in\Gamma_0(N)\backslash\Gamma_0(M)} E^{\scriptscriptstyle{(N)}}_{\mathfrak{a}}(\gamma \gamma_{\mathfrak{b}} W z,s) = \sum_{\gamma\in\Gamma_0(N)\backslash\Gamma_0(M)} E^{\scriptscriptstyle{(N)}}_{\mathfrak{a}} \Big(\sigma_{\gamma\mathfrak{b},\scriptscriptstyle{N}} \tfrac{W}{W^1_N(\gamma\mathfrak{b})}z, s \Big) \\
    &= \sum_{\gamma\in\Gamma_0(N)\backslash\Gamma_0(M)} \delta_{\gamma \mathfrak{b} \overset{\scriptscriptstyle N}{=} \mathfrak{a}} \Big(\tfrac{W}{W^1_N(\gamma\mathfrak{b})} y \Big)^s + O(1)\\
    &= \delta_{\mathfrak{b} \overset{\scriptscriptstyle M}{=} \mathfrak{a}} W^M_N(\mathfrak{a}) \Big(\tfrac{W^1_M(\mathfrak{a})}{W^1_N(\mathfrak{a})} y \Big)^s + O(1) = \delta_{\mathfrak{b} \overset{\scriptscriptstyle M}{=} \mathfrak{a}} W^M_N(\mathfrak{a})^{1-s} y^s +O(1).
\end{align*}
On the other hand, $(W^M_N({\mathfrak{a}}))^{1-s} E_{\mathfrak{a}}^{\scriptscriptstyle{(M)}}|_{\sigma_{\mathfrak{b}}}$ has exactly the same formula as above by (\ref{ab}), which finishes the proof.
\end{proof}

\subsection{Explicit calculations with scattering matrices and related quantities}
As is mentioned in Section \ref{section:firstresults}, we need to study the behavior of $|E_{\infty}(z,s,\chi)|^2$ at each cusp in $\mathcal{C}(N)$,  not just these in $\mathcal{C}_{\chi}(N)$.
The change-of-basis formula, Theorem \ref{c2c}, now helps.
\subsubsection{Preparation} We begin with proving a lemma.
\begin{mlemma}\label{ytothes}
Let $K \geq 1$, and $\gamma = (\begin{smallmatrix} u & v \\ f & w \end{smallmatrix}) \in SL_2(\mathbb{Z})$ with $f\mid N$.
Then there exist meromorphic $C_{\chi_1, \chi_2}(s)$ and $D_{\chi_1, \chi_2}(s)$ (depending on $K$ and $\gamma$) such that
\begin{align}\label{constantform}
 E_{\chi_1,\chi_2}(K \gamma z, s) = C_{\chi_1, \chi_2}(s) y^s + D_{\chi_1, \chi_2}(s)y^{1-s} + o(1),
\end{align}
as $y\rightarrow\infty$.
Precisely, %(see (\ref{theta}) for expression of $\theta_{\chi_1,\chi_2}$)
\begin{align}\label{CD}
\begin{split}
C_{\chi_1, \chi_2}(s) &= \delta_{q_2 \mid f}\frac{(q_2 K, f)^{2s}}{q_2^s K^s} \chi_1 \Big(\frac{-f}{(q_2 K, f)} \Big) \chi_2 \Big( \frac{q_2 K u}{(q_2 K, f)} \Big), \\
D_{\chi_1, \chi_2}(s) &= \delta_{q_1 \mid f} \frac{\theta_{\overline{\chi_2},\overline{\chi_1}}(1-s)}{\theta_{\chi_1,\chi_2}(s)} \frac{(q_1 K, f)^{2-2s}}{q_1^{1-s} K^{1-s}} \overline{\chi_1}\Big(\frac{q_1 K u}{(q_1 K, f)} \Big) \overline{\chi_2} \Big(\frac{-f}{(q_1 K, f)} \Big).
\end{split}
\end{align}
\end{mlemma}
\begin{proof}
Observe $E_{\chi_1,\chi_2}(K\gamma z,s)$ is periodic with some integer period. By \cite[Proposition 1.5]{I2}, (\ref{constantform}) holds. To obtain (\ref{CD}), we proceed directly.
By definition, we have

\begin{align*}
    E_{\chi_1,\chi_2}(K \gamma z, s)  &= \frac{1}{2}\sum_{(c,d)=1} \frac{(q_2 \Im (K \gamma z))^s \chi_1(c)\chi_2(d)}{ | c q_2 K \gamma z + d|^{2s}} \\
    = \frac{1}{2} \sum_{(c,d)=1} & \frac{(q_2 Ky)^s\chi_1(c)\chi_2(d)}{ |(c q_2 K u + d f)z +(c q_2 K v + d w)|^{2s}} 
    = \frac{1}{2} \sum_{\ell \in \mathbb{Z}} \sum_{\substack{(c,d)=1 \\ c q_2 K u +d f = \ell}} \frac{(q_2 Ky)^s \chi_1(c)\chi_2(d)}{|\ell z + (c q_2 K v + d w) |^{2s}}.
\end{align*}

 For any $\Re s>1$, we see that as $y \rightarrow\infty$, uniform convergence allows us to interchange the limit and the sums, yielding
\begin{align*}
   E_{\chi_1,\chi_2}(K \gamma z, s) = C(s) y^s + o(1),
%\end{align*}
\quad \text{for} \quad
%\begin{align*}
    C(s)= \frac{1}{2} \sum_{\substack{(c,d)=1 \\ c q_2 K u +d f = 0}} \frac{(q_2 K)^s \chi_1(c)\chi_2(d)}{|c q_2 K v + d w|^{2s}}.
\end{align*}
Then (\ref{constantform}) implies that $C(s)=C_{\chi_1, \chi_2}(s)$, and we can calculate $C_{\chi_1, \chi_2}(s)$ by simplifying the above expression. 
Solving $c q_2 K u +d f = 0$ for $(c,d)=1$ and $\chi_1(c)\chi_2(d) \neq 0$, we can easily see the solutions exist only if $q_2 \mid f$, and they are
\begin{align*}
\begin{cases} c=\pm \frac{f}{(q_2 K, f)} \\ d=\mp \frac{q_2 K u}{(q_2 K, f)}. \end{cases}
\end{align*}
Since $uw-vf=1$ and $\chi_1\chi_2(-1)=1$, we arrive at the desired expression for $C_{\chi_1,\chi_2}(s)$. By Proposition \ref{FE}, we have
\begin{align*}
    D_{\chi_1, \chi_2}(s) = \tfrac{\theta_{\overline{\chi_2},\overline{\chi_1}}(1-s)}{\theta_{\chi_1,\chi_2}(s)} C_{\overline{\chi_2},\overline{\chi_1}}(1-s).
\end{align*}
Inserting the formula of $C_{\overline{\chi_2},\overline{\chi_1}}$, we complete the proof.
\end{proof}

\subsubsection{Entries of scattering matrices}
Recall $\Phi(s,\chi)$ in Proposition \ref{Selberg}, which is a matrix of entries $\varphi_{\mathfrak{ab}}(s,\chi)$ for all $\mathfrak{a,b}\in \mathcal{C}_{\chi}(N)$. With Lemma \ref{ytothes}, we can write out every entry of $\Phi$. The logarithmic derivative of the determinant of scattering matrices are important for their occurrence in the Selberg Trace Formula.

\begin{mprop}\label{scatteringmatrix}
If $\mathfrak{a,b}\in \mathcal{C}_{\chi}(N)$, then
\begin{multline*}
\varphi_{\mathfrak{ab}}(s,\chi) = \frac{f_{\mathfrak{a}}^{-1}W_{\mathfrak{a}}^{-s} W_{\mathfrak{b}}^{1-s}}  
{\varphi((f_{\mathfrak{a}},\frac{N}{f_{\mathfrak{a}}}))} \sum_{q_1 \mid (\frac{N}{f_{\mathfrak{a}}}, f_{\mathfrak{b}})} \sum_{q_2 \mid f_{\mathfrak{a}}} \sideset{}{^*}\sum_{\chi_1, \chi_2}\overline{\chi_1}(u_{\mathfrak{b}})\overline{\chi_2}(u_{\mathfrak{a}}) \frac{L(2s, \chi_1\chi_2)}{L(2s,\chi_1\chi_2\chi_{_0,_N})} \frac{\theta_{\overline{\chi_2},\overline{\chi_1}}(1-s)}{\theta_{\chi_1,\chi_2}(s)} \\
(\tfrac{q_2}{q_1} )^{1-s} \sum_{a\mid f_{\mathfrak{a}}} \sum_{b \mid \frac{N}{f_{\mathfrak{a}}}} \frac{\mu(a)\mu(b)\chi_1(b)\chi_2(a)}{a^{2s-1} b} (q_1 \tfrac{bf_{\mathfrak{a}}}{a q_2}, f_{\mathfrak{b}})^{2-2s}
\overline{\chi_1}\Big(\frac{q_1 \frac{bf_{\mathfrak{a}}}{a q_2}}{(q_1 \frac{bf_{\mathfrak{a}}}{a q_2}, f_{\mathfrak{b}})} \Big) \overline{\chi_2} \Big(\frac{f_{\mathfrak{b}}}{(q_1 \frac{bf_{\mathfrak{a}}}{a q_2}, f_{\mathfrak{b}})} \Big),
\end{multline*}
where the asterisked sum is over all primitive $\chi_i \pmod{q_i}$ for $i=1,2$ with $\chi_1\overline{\chi_2}\simeq \chi$ (see Convention \ref{simeq} for definition).
\end{mprop}

\begin{proof}
For $\mathfrak{b}=\frac{u_{\mathfrak{b}}}{f_{\mathfrak{b}}}$ as is in (\ref{DS}), we have by Theorem \ref{c2c}
\begin{multline*}
\varphi_{\mathfrak{ab}}(s,\chi) = \frac{f_{\mathfrak{a}}^{-s} W_{\mathfrak{a}}^{-s}}{\varphi((f_{\mathfrak{a}},\frac{N}{f_{\mathfrak{a}}}))} \sum_{q_1\mid \frac{N}{f_{\mathfrak{a}}}} \sum_{q_2 \mid f_{\mathfrak{a}}} \sideset{}{^*}\sum_{\chi_1, \chi_2}\overline{\chi_2}(-u_{\mathfrak{a}})\frac{L(2s, \chi_1\chi_2)}{L(2s,\chi_1\chi_2\chi_{_0,_N})} \\
\sum_{a\mid f_{\mathfrak{a}}} \sum_{b \mid \frac{N}{f_{\mathfrak{a}}}} \frac{\mu(a)\mu(b)\chi_1(b)\chi_2(a)}{(ab)^s} \Psi\Big(E_{\chi_1, \chi_2} \Big(\frac{bf_{\mathfrak{a}}}{a q_2} \sigma_{\mathfrak{b}}z, s \Big)\Big),
\end{multline*}
where $\Psi(E_{\chi_1,\chi_2})$ stands for the coefficient of the $y^{1-s}$-term of $E_{\chi_1,\chi_2}$. Since the choice of $\sigma_{\mathfrak{b}}$ does not affect the constant term in the Fourier expansion,  we can take
\begin{align*}
\sigma_{\mathfrak{b}} = \gamma_{\mathfrak{b}} \Big(\begin{smallmatrix} W_{\mathfrak{b}}^{1/2} & 0 \\ 0 & W_{\mathfrak{b}}^{-1/2} \end{smallmatrix}\Big)
\end{align*}
by Remark \ref{gamma2sigma}, where $\gamma_{\mathfrak{b}} = (\begin{smallmatrix} u_{\mathfrak{b}} & v \\ f_{\mathfrak{b}} & w \end{smallmatrix})\in SL_2(\mathbb{Z})$. 
Then for $K=\frac{bf_{\mathfrak{a}}}{a q_2}$, and $\gamma=\gamma_{\mathfrak{b}}$, (\ref{CD}) gives
\begin{multline*}
\Psi\Big(E_{\chi_1, \chi_2} \Big(\frac{bf_{\mathfrak{a}}}{a q_2} \sigma_{\mathfrak{b}}z, s \Big)\Big) =  \delta_{q_1 \mid f_{\mathfrak{b}}} \frac{\theta_{\overline{\chi_2},\overline{\chi_1}}(1-s)}{\theta_{\chi_1,\chi_2}(s)} \frac{(q_1 \frac{bf_{\mathfrak{a}}}{a q_2}, f_{\mathfrak{b}})^{2-2s}}{q_1^{1-s} (\frac{bf_{\mathfrak{a}}}{a q_2})^{1-s}} \\
\overline{\chi_1}\Big(\frac{u_{\mathfrak{b}} q_1 \frac{bf_{\mathfrak{a}}}{a q_2}}{(q_1 \frac{bf_{\mathfrak{a}}}{a q_2}, f_{\mathfrak{b}})} \Big) \overline{\chi_2} \Big(\frac{- f_{\mathfrak{b}}}{(q_1 \frac{bf_{\mathfrak{a}}}{a q_2}, f_{\mathfrak{b}})} \Big) W_{\mathfrak{b}}^{1-s}.
\end{multline*}
Then we complete the proof after substitution.
\end{proof}

There are two special cases of Proposition \ref{scatteringmatrix} of special interest in this paper. 

Firstly, we consider the case $\mathfrak{a}=\infty$. Notice that $\big( \begin{smallmatrix} 1&0\\N&1 \end{smallmatrix}\big) \mathfrak{a} = \mathfrak{a}'=\frac{1}{N}$, by Remark \ref{cb}, so we have $\varphi_{\mathfrak{ab}} = \chi(1) \varphi_{\mathfrak{a}'\mathfrak{b}} = \varphi_{\mathfrak{a}'\mathfrak{b}}$. In addition, we have the following closed-form formula.
\begin{mcoro}\label{varphi}
For $\mathfrak{b}=\frac{u}{f} \in \mathcal{C}_{\chi}(N)$ in (\ref{DS}), we have 
\begin{align*}
    \varphi_{\infty\mathfrak{b}}(s,\chi)= \delta_{f\mid \frac{N}{q}} \tau(\overline{\psi}) \frac{ W_{\mathfrak{b}}^{-s} f^{1-2s}}{\varphi((f,\frac{N}{f}))} \frac{\Lambda(2-2s,\psi)}{\Lambda(2s,\overline{\psi})} \prod_{p\mid N} \Big(1-\frac{\overline{\psi}(p)}{p^{2s}} \Big)^{-1} \prod_{p\mid f}(1-\frac{1}{p}) \prod_{p\mid \frac{N}{f}} \Big(1-\frac{\overline{\psi}(p)}{p^{2s-1}} \Big),
\end{align*}
where $\Lambda$ is the completed Dirichlet $L$-function. In particular, $\varphi_{\infty\infty}(s,\chi)=0$ if $\chi \neq \chi_{_0,_N}$, and $\varphi_{\infty\mathfrak{a}}(\tfrac{1}{2})=-\delta_{\mathfrak{a}\overset{\scriptscriptstyle N}{=}\infty}$.
\end{mcoro}

\begin{proof}[Sketch of proof]
We need to substitute $f_{\mathfrak{a}}=N$, $f_{\mathfrak{b}}=f$ into Proposition \ref{scatteringmatrix}. Briefly, after some local analysis over different types of prime numbers, we have 
\begin{align*}
\sum_{a\mid N} \frac{\mu(a)\overline{\psi}(a)}{a^{2s-1}} \Big(\frac{N}{aq},f\Big)^{2-2s} \psi\Big(\frac{f}{(\frac{N}{aq},f)}\Big) = \delta_{f\mid \frac{N}{q}} f^{2-2s} \prod_{p\mid \frac{N}{f}}\Big(1-\frac{\overline{\psi}(p)}{p^{2s-1}}\Big) \prod_{p\nmid \frac{N}{f}}\Big(1-\frac{1}{p}\Big).
\end{align*}
One can verify the rest easily and complete the proof.
\end{proof}
Secondly, we assume $\chi$ is primitive $\negthinspace \negthinspace \pmod{N}$, where only Atkin--Lehner cusps are singular for $\chi$. Given an Atkin--Lehner cusp $\mathfrak{a}=\frac{1}{f}\in \mathcal{C}(N)$, we call 
$\mathfrak{a}^*:=\frac{1}{N/f}\in \mathcal{C}(N)$ the \textit{Atkin--Lehner complement of} $\mathfrak{a}$ (\textit{on level} $N$). The following calculation by N. Pitt depicts a special property of Atkin--Lehner complement. Humphries \cite{Hum1} computed it in full details, and on general weights.

\begin{mcoro}\cite[Proposition 13.7]{I1}\label{pitt}
If $\mathfrak{a,b}\in\mathcal{C}(N)$ are Atkin--Lehner, and $\chi=\chi_1 \overline{\chi_2}$ with $\chi_1$ primitive $\mymod{\frac{N}{f_{\mathfrak{a}}}}$ and $\chi_2$ primitive $\mymod{f_{\mathfrak{a}}}$, then we have 
\begin{align*}
 \varphi_{\mathfrak{ab}}(s,\chi) = \begin{cases}\chi_1(-1)\tau(\chi_1)\tau(\chi_2) N^{-s} \frac{\Lambda(2-2s,\overline{\chi_1 \chi_2})}{\Lambda(2s,\chi_1 \chi_2 )}  & \text{ if } \mathfrak{b}=\mathfrak{a}^*; \\
 0 & \text{ otherwise}. \end{cases}
\end{align*}
\end{mcoro}

\subsubsection{The behavior of Eisenstein series at cusps that are not singular}
As we have mentioned in Section \ref{section:firstresults}, the cuspidal behavior of Eisenstein series at cusps not singular for the central character affects the precise description of $\mathcal{E}$. 
\begin{mprop}\label{nonsingular}
If $\mathfrak{a}\in \mathcal{C}_{\chi}(N)$, and $\mathfrak{b} \in \mathcal{C}(N)\backslash \mathcal{C}_{\chi}(N)$, then as $y\rightarrow\infty$, we have
\begin{align*}
    E_{\mathfrak{a}}(\sigma_{\mathfrak{b}}z,s,\chi) = o_s(1).
\end{align*}

\end{mprop}
Selberg proved (yet not published) the proposition for primitive $\chi$; see \cite[Thm.\ 7.1, p.641]{Se}.
Here we give an alternative proof, for which we need some preparation. 
\begin{mconv}\label{nu}
We denote the $p$-adic order function by $\nu_p(\cdot)$.
\end{mconv}

\begin{mlemma}\label{zeroconstc}
Let $\chi_i$ be primitive $\mymod{q_i}$ for $i=1,2$, and $\chi=\chi_1\overline{\chi_2}$ be induced by primitive $\psi \pmod{q}$.  Assume there is $f\mid N$ such that $q_1 \mid \frac{N}{f}$ and $q_2 \mid f$, and $K\mid N$ satisfying:
\begin{align}\label{Krestriction}
\nu_p(K) \leq
\begin{cases}
\nu_p(N) - \nu_p(q_2) & \text{ if } p\nmid q_1, p\mid q_2; \\
\nu_p(f) - \nu_p(q_2) & \text{ if } p\mid q_1. 
\end{cases}
\end{align}
If $E_{\chi_1,\chi_2}(K\sigma_{\mathfrak{b}}z,s,\chi)$
is unbounded as $y\rightarrow \infty$ for some $\mathfrak{b}\in \mathcal{C}(N)$, then $\mathfrak{b} \in \mathcal{C}_{\chi}(N)$.
\end{mlemma}
\begin{proof}
If $E_{\chi_1,\chi_2}|_{K\sigma_{\mathfrak{b}}}$ is unbounded, then by Lemma \ref{ytothes}, either $C_{\chi_1,\chi_2}(s)\neq 0$ or $D_{\chi_1,\chi_2}(s)\neq 0$.

In the \textbf{former case}, we have $q_2\mid f_{\mathfrak{b}}$, and for all prime numbers $p\mid q_1$,
\begin{align*}
    \nu_p(K) \geq \nu_p(f_{\mathfrak{b}}) - \nu_p(q_2).
\end{align*}
From (\ref{Krestriction}), we know $\nu_p(K) \leq \nu_p(f) - \nu_p(q_2)$, which gives $\nu_p(f) \geq  \nu_p(f_{\mathfrak{b}})$. Then by assumption on $f$, we have
\begin{align*}
   \nu_p(q_1) \leq \nu_p(N/f) \leq \nu_p(N/f_{\mathfrak{b}}),
\end{align*}
indicating $q_1\mid \frac{N}{f_{\mathfrak{b}}}$. Together with $q_2 \mid f_{\mathfrak{b}}$, we find $q=[q_1,q_2] \mid [f_{\mathfrak{b}}, \frac{N}{f_{\mathfrak{b}}}]$, which means $\mathfrak{b}$ is singular for $\chi$ by Proposition \ref{criterion}.

In the \textbf{latter case}, we have $q_1 \mid f_{\mathfrak{b}}$, and for all prime numbers $p\mid q_2$, \begin{align*}
    \nu_p(f_{\mathfrak{b}}) \leq \nu_p(q_1) + \nu_p(K).
\end{align*}
We want to show 
\begin{align}\label{pmidq2}
\nu_p(q_2) \leq \nu_p(N) - \nu_p(f_{\mathfrak{b}})
\end{align}
for all $p\mid q_2$, since this implies $q_2\mid \frac{N}{f_{\mathfrak{b}}}$, and hence that $\mathfrak{b}$ is singular for $\chi_1\overline{\chi_2}$ for the same reason in the previous case.
We further bifurcate the discussion.
Say $p$ also divides $q_1$.  Then
\begin{align*}
\nu_p(f_{\mathfrak{b}}) \leq 
\nu_p(q_1) + \nu_p(K) \leq
\nu_p(q_1) + \nu_p(f) - \nu_p(q_2) \leq \nu_p(N) - \nu_p(q_2),
\end{align*}
Thus \eqref{pmidq2} holds.
On the contrary, if $p \nmid q_1$, then $\nu_p(f_{\mathfrak{b}}) \leq \nu_p(K) \leq \nu_p(N) - \nu_p(q_2)$, giving \eqref{pmidq2} again.
\end{proof}

\begin{proof}[Proof of Proposition \ref{nonsingular}]
By Theorem \ref{c2c}, $E_{\mathfrak{a}}(\sigma_{\mathfrak{b}}z,s,\chi)$ equals a linear combination of $E_{\chi_1,\chi_2}(K\sigma_{\mathfrak{b}}z,s,\chi)$, where $\chi_i$ is primitive $\mymod{q_i}$ for $i=1,2$, $\chi_1\overline{\chi_2} \simeq \chi$, and $K \mid N$ satisfies (\ref{Krestriction}). By Lemma \ref{zeroconstc}, none of these $E_{\chi_1,\chi_2}(K\sigma_{\mathfrak{b}}z,s,\chi)$ contributes any $y^s$ or $y^{1-s}$-terms, so the proof is complete.
\end{proof}

\subsection{The formal inner product of Eisenstein series}\label{truncated}
It is well-known that Eisenstein series are not in $L^2$. 
It is nevertheless useful to
consider the \textit{formal inner product} of two Eisenstein series inspired by \cite[Section 7.1]{I2}. Concretely, if $\mathfrak{a,b} \in \mathcal{C}(N)$, then the formal inner product of $E_{\mathfrak{a}}$ and $E_{\mathfrak{b}}$ is defined by
\begin{align*}
\langle E_{\mathfrak{a}}(\cdot,s), E_{\mathfrak{b}}(\cdot,s) \rangle_{_N}^{\text{Eis}} := 4\pi \delta_{\mathfrak{ab}},
\end{align*}
when $s=\frac{1}{2}+iT$.
For more details, see Section \ref{spec}, where we adopt newform Eisenstein series to build an alternative orthonormal basis. To accomplish this, we have the following lemma as a special case of \cite[Lemma 8.3]{Y2}.

\begin{mlemma}\label{renormalization}
For primitive $\psi \pmod{q}$ with $q^2 \mid N$, we have
\begin{align*}
\langle E_{\psi,\psi}, E_{\psi,\psi} \rangle_{_N}^{\mathrm{Eis}}= 4\pi N \prod_{p\mid q}(1-p^{-1}) \prod_{p \mid N}(1+ \chi_{_0,_q}(p) p^{-1}).
\end{align*}
\end{mlemma}

\subsection{Laurent expansions of Eisenstein series}

\begin{mprop}\label{Laurent}
There is an $SL_2(\mathbb{Z})$-invariant function $G$ such that
\begin{align*}
E(z,s) = \frac{3/\pi}{s-1} + G(z) + O(|s-1|),
\end{align*}
and as $y\rightarrow \infty$,
\begin{align}
\label{eq:Gasymptotic}
G(z) = y + O(\log y).
\end{align} 
\end{mprop}
Proposition \ref{Laurent} follows directly from \cite[(22.66)--(22.69)]{IK}, so we omit the proof. These formulas also show that $G(z)\in \mathcal{A}(Y_0(1))$  can be expressed in terms of the logarithm of the Dedekind eta function, but all we need for our later purposes is \eqref{eq:Gasymptotic}.
  
It is also important to explicitly evaluate the Laurent expansion of $E_{\mathfrak{a}}(z,s)$ around $s=1$ in terms of the newform Eisenstein series.
\begin{mprop}\label{cuspLaurent}
For $\mathfrak{a}=\frac{u}{f}\in \mathcal{C}(N)$, we have
\begin{multline*}
    E_{\mathfrak{a}}(z,s) =  \frac{\Vol(Y_0(N))^{-1}}{s-1} + c_{\mathfrak{a},0} + \sum_{g\mid N} c_{\mathfrak{a},g} G|_g \\
+ \sum_{1<r\mid (f,N/f)}\sideset{}{^*}\sum_{\eta (r)} \overline{\eta}(u) \sum_{g \mid Nr^{-2}} c_{\mathfrak{a}, \eta, g}E_{\eta,\eta}(gz,1) + O(|s-1|),
\end{multline*}
where $c_{\mathfrak{a}, \eta, g}$ are independent of $u$,
\begin{align}\label{ca0}
c_{\mathfrak{a},0} = \frac{1}{\Vol(Y_0(N))} \Big(\log \frac{(f,\tfrac{N}{f})}{N} + \sum_{p\mid N} \frac{\log p}{p+1} - \sum_{p\mid (f,N/f)} \frac{\log p}{p-1} \Big),
\end{align}
 and
 \begin{equation}
 \label{eq:cagSummedOverAandBversion}
 c_{\mathfrak{a}, g} = \frac{(f,N/f)}{N \varphi((f,N/f))} \frac{\zeta(2)}{L(2, \chi_{0,N})} 
 \mathop{\sum_{a|f} \sum_{b|\frac{N}{f}}}_{} \delta_{bf/a=g} \frac{\mu(a) \mu(b)}{ab}.
\end{equation}
\end{mprop}
\begin{proof}
By Theorem \ref{c2c}, $E_{\mathfrak{a}}(z,s)$ can be expressed as a linear combination of $E_{\eta,\eta}|_g$ for primitive $\eta$ (mod $r$) with $r\mid (f, N/f)$, and suitable $g|N$. The contribution from $r>1$ is
\begin{equation*}
\frac{W_{\mathfrak{a}}^{-s} f^{-s}}{\varphi((f,N/f))}  \sum_{1<r\mid (f,N/f)} \sideset{}{^*}\sum_{\eta} \overline{\eta}(-u) \frac{L(2s, \eta^2)}{L(2s, \eta^2 \chi_{_0,_N})}
\sum_{a\mid f}\sum_{b\mid \frac{N}{f}} \frac{\mu(a)\mu(b)\eta(ab)}{a^s b^s} E_{\eta,\eta}\Big(\frac{b f}{a r} z,s\Big),
\end{equation*}
which can be expressed as $\sum_{1<r\mid (f,N/f)}\sum^*_{\eta (r)} \overline{\eta}(-u) \sum_{g \mid Nr^{-2}} c_{\mathfrak{a}, \eta, g}E_{\eta,\eta}(gz,1)$ with $c_{\mathfrak{a}, \eta, g}$ independent of $u$. 
By Proposition \ref{Laurent}, the contribution from $r=1$ equals
\begin{equation*}
\frac{W_{\mathfrak{a}}^{-s} f^{-s}}{\varphi((f,N/f))} \frac{\zeta(2s)}{L(2s,\chi_{_0,_N})} \sum_{a\mid f}\sum_{b\mid \frac{N}{f}} \frac{\mu(a)\mu(b)}{a^s b^s} \Big(\frac{3/\pi}{s-1} + G\Big(\frac{bf}{a} z\Big) + O(s-1)\Big).
%E\Big(\frac{b f}{a} z,s\Big).
\end{equation*}
Let
\begin{equation}
\label{eq:Fdef}
 F_{\mathfrak{a}}(s) = \frac{W_{\mathfrak{a}}^{-s} f^{-s}}{\varphi((f,N/f))} \frac{\zeta(2s)}{L(2s,\chi_{_0,_N})} \sum_{a\mid f}\sum_{b\mid \frac{N}{f}} \frac{\mu(a)\mu(b)}{a^s b^s}.
\end{equation}
It is well-known that $\text{Res}_{s=1} E_{\mathfrak{a}}(z,s) = (\Vol(Y_0(N)))^{-1}$, so $\frac{3}{\pi} F_{\mathfrak{a}}(1) = \Vol(Y_0(N))^{-1}$; of course, for consistency this can be checked directly from \eqref{eq:Fdef}.  Hence the contribution of $r=1$ to the Laurent expansion of $E_{\mathfrak{a}}(z,s)$ is of the form
\begin{equation}\label{generallaurent}
 \frac{\Vol(Y_0(N))^{-1}}{s-1} + \tfrac{3}{\pi} F_{\mathfrak{a}}'(1) + \sum_{g|N} c_{\mathfrak{a}, g} G \vert_g  +O(s-1),
\end{equation}
for $c_{\mathfrak{a}, g}$ given by \eqref{eq:cagSummedOverAandBversion}.  
The term $F_{\mathfrak{a}}'(1)$ gives rise to $c_{\mathfrak{a},0}$, which is computed by

\begin{equation*}
\frac{F_{\mathfrak{a}}'}{F_{\mathfrak{a}}}(1) =  -\log N + \log (f, \tfrac{N}{f}) +  \sum_{p\mid N} \frac{\log p}{p+1} - \sum_{p\mid (f,N/f)} \frac{\log p}{p-1}.
\qedhere
\end{equation*}
\end{proof}

Although the level $1$ Eisenstein series is an eigenfunction of the Hecke operators, the same is not quite true for the function $G$.
\begin{mlemma}\label{TnG}
For $n\geq 1$, we have 
\begin{align*}
    T_n (G) = \lambda(n) G + \frac{3}{\pi}\sqrt{n}\sum_{a\mid n} a^{-1}\log \frac{n}{a^2},
\end{align*}
where $T_n$ is the $n$-th Hecke operator, and $\lambda(n)=\lambda_{1,1}(n,1) = n^{1/2} \sum_{b|n} b^{-1}$ as is in (\ref{specialf=N}).
\end{mlemma}
\begin{mrema}
Our normalization of the Hecke operator $T_n$ is so that $T_n u_j = \lambda_j(n) u_j$ (and see Convention \ref{conv1}).
\end{mrema}

\begin{proof}
Recall that $G(z) = \text{Res}_{s=1} (s-1)^{-1} E(z,s)$, so by Remark \ref{Hecke}
we have
\begin{align*}
    T_n(G)  = \underset{{s=1}}{\text{Res}} \Big( (s-1)^{-1} \lambda(n,s) E(z,s) \Big).
\end{align*} 
By Proposition \ref{Laurent} and since $\lambda(n,s) = \sum_{ab=n} (\frac{b}{a})^{s-1/2}$, we finish the proof.
\end{proof}

\subsection{Some inequalities}
Here we perform some elementary calculations related to $\varphi_{\infty\mathfrak{a}}$, which are critical for future arguments. To begin, we have the following standard lemma. 
\begin{mlemma}\label{Mertens}
There exists an absolute constant $C$ so that
\begin{align*}
\sum_{p \mid N} \frac{1}{p} \leq \log\log\log (N + 15) + C, \quad \text{ and } \quad
\sum_{p \mid N} \frac{\log p}{p} \leq \log\log (N + 2) + C.
\end{align*}
\end{mlemma}
\begin{mconv}\label{fulless} For integers $A$ and $B$, we denote $\lim_{n\rightarrow \infty} (A, B^n)$ by $A_B$ and $\tfrac{A}{A_B}$ by $A_B^{\scriptscriptstyle{\perp}}$.\end{mconv}
From the fact $N_q^{\scriptscriptstyle{\perp}} \mid \tfrac{N}{q}$, we have the following corollary.
\begin{mcoro}\label{easypart}
If $s=\frac{1}{2}+iT$ and $\psi$ is primitive $\negthinspace \negthinspace \pmod{q}$ for $q\mid N$, then
\begin{align*}
\sum_{p\mid N} \frac{\psi(p) \log p}{\psi(p)p^{2s}-1} \ll \log \log (\tfrac{N}{q}+2).
\end{align*}
\end{mcoro}
Then we can bound the coefficients in Proposition \ref{cuspLaurent}.
\begin{mcoro}\label{boundsofcuspLaurent}
For $\mathfrak{a}=\frac{u}{f}\in \mathcal{C}(N)$, we have
\begin{align}
\label{eq:ca0estimate}
c_{\mathfrak{a},0} =  \frac{1}{\Vol(Y_0(N))} \Big(\log \frac{(f,N/f)}{N} + O(\log\log N) \Big),
\end{align}
and
\begin{align}\label{eq:cagestimate}
\sum_{g\mid N}|c_{\mathfrak{a},g}| \ll N^{-1} (\log\log N)^3.
\end{align}
\end{mcoro}
\begin{proof}
The equation \eqref{eq:ca0estimate} follows from Lemma \ref{Mertens}. By (\ref{eq:cagSummedOverAandBversion}), we have
\begin{align*}
\sum_{g\mid N}|c_{\mathfrak{a},g}| &\leq \frac{(f,N/f)}{N \varphi((f,N/f))} \frac{\zeta(2)}{L(2, \chi_{0,N})} 
  \mathop{\sum_{a|f} \sum_{b|\frac{N}{f}}} \frac{|\mu(a) \mu(b)|}{ab}  \\
&= N^{-1} \prod_{p\mid (f,N/f)}(1-p^{-1})^{-1} \prod_{p\mid N}(1-p^{-2})^{-1} \prod_{p\mid f} (1+p^{-1}) \prod_{p\mid \frac{N}{f}} (1+p^{-1}).
\end{align*}
Then Lemma \ref{Mertens} completes the proof of \eqref{eq:cagestimate}.
\end{proof}
\begin{mconv}
Given $n \geq 1$, we denote the number of prime divisors of $n$ by $\omega(n)$. 
\end{mconv}
\begin{mprop}\label{hardcore}
For any positive integers $k$ and $L$,
\begin{align*}
\sum_{g \mid L}\frac{\log g}{g} k^{\omega(g)} \ll_k (\log \log (L + 2))^{k+1}.
\end{align*}
\end{mprop}
\begin{proof}
Decomposing $\log g$ into $\sum_{p\mid g} \nu_p(g)\log p$, we have
\begin{multline*}
\sum_{g\mid L} \frac{\log g}{g} k^{\omega(g)} = \sum_{p\mid L} \log p \sum_{\substack{g \mid L \\ g\equiv 0 (p)}}\frac{\nu_p(g)}{g} k^{\omega(g)} = \sum_{p\mid L} \log p \sum_{i=1}^{\nu_p(L)}i \sum_{\substack{g \mid L \\ \nu_p(g)=i}} \frac{k^{\omega(g)}}{g}\\
= \sum_{p\mid L}\log p \sum_{i=1}^{\nu_p(L)} \frac{ik}{p^i} \sum_{\substack{g \mid L \\ g \not\equiv 0 (p)}} \frac{k^{\omega(g)}}{g} = \underbrace{k\sum_{p\mid L}\log p \sum_{i=1}^{\nu_p(L)} \frac{i}{p^i}}_{A} \overbrace{\prod_{\substack{p' \mid L \\ p' \neq p}} \big( 1 + k \sum_{j=1}^{\nu_{p'}(L)} \frac{1}{(p')^j} \big)}^{B(p)}.
\end{multline*}
It is not hard to find that $0 < A \ll \sum_{p\mid L}\frac{\log p}{p} \ll \log \log (L +2)$ by Lemma \ref{Mertens}. Since $1 \leq B(p)\leq \prod_{p' \mid L} \big( 1 + k \sum_{j=1}^{\infty} \frac{1}{(p')^j} \big) =: B$, we have again by Lemma \ref{Mertens}
\begin{align*}
\log B = \sum_{p \mid L} \log \big( 1 + k \sum_{j=1}^{\infty} \frac{1}{p^j} \big) = k \sum_{p\mid L} \frac{1}{p} + O_k(1) \leq k \log\log\log (L +15) + O_k(1).
\end{align*}
Then $B \ll_k (\log\log (L+2))^k$ implies $\sum_{g \mid L}\frac{\log g}{g} k^{\omega(g)} \leq AB \ll_k (\log\log (L +2))^{k+1}$.
\end{proof}
\begin{mcoro}\label{hardpart}
For $\mathfrak{a}=\frac{u_{\mathfrak{a}}}{f_{\mathfrak{a}}}\in \mathcal{C}(N)$ as in (\ref{DS}), and $s = \frac{1}{2}+iT$, we have
\begin{align}
        &\sum_{\mathfrak{a}} |\varphi_{\infty\mathfrak{a}}(s, \chi)|^2 \log \frac{N}{q f_{\mathfrak{a}}} \ll \Big(\log\log \Big(\frac{N}{q} +2\Big) \Big)^5 ; \label{hardpart1}\\
        &\sum_{\mathfrak{a}}|\varphi_{\infty\mathfrak{a}}(s, \chi)|^2 \sum_{p\mid \frac{N}{f_{\mathfrak{a}}}}\frac{\psi(p) \log p}{\psi(p) p^{2s-1}-1} \ll \Big(\log\log \big(\frac{N}{q} +2\big) \Big)^5 ; \label{hardpart2}\\
        &\sum_{\mathfrak{a}}|\varphi_{\infty\mathfrak{a}}(s, \chi)|^2 \log f_{\mathfrak{a}} = \log \frac{N}{q} + O\Big( \Big(\log\log \big(\frac{N}{q} +2\big) \Big)^5 \Big). \label{hardpart3}
\end{align}
\end{mcoro}

\begin{proof}
Define $S_f(s,\chi)
    := \sum_{\mathfrak{a}: f_{\mathfrak{a}}=f}|\varphi_{\infty \mathfrak{a}}(s,\chi)|^2$ for $f\mid \frac{N}{q}$.  By Corollary \ref{varphi}, we have
\begin{align*}
     S_f(s,\chi) = C_f(s,\chi) \prod_{p\mid \frac{N}{f}}S_f^p(s,\chi),
\end{align*}
where 
\begin{align*}
    C_f(s,\chi)=\frac{q f}{N}\prod_{p\mid (f,\frac{N}{f})}(1-p^{-1}) \prod_{p\mid N_{(N/f)}^{\scriptscriptstyle{\perp}}}|1-\overline{\psi}(p)p^{-2s}|^{-2}(1-p^{-1})^2 \leq \frac{q f}{N},
\end{align*}
and 
\begin{align*}
    S_f^p(s,\chi)= \Big|\frac{1-\overline{\psi}(p)p^{1-2s}}{1-\overline{\psi}(p)p^{-2s}} \Big|^2 \leq \begin{cases} 4 & \text{ if } p \nmid q, \\ 1 & \text{ if } p \mid q.  \end{cases}
\end{align*}
There being at most $\omega((\frac{N}{f})_{q}^{\scriptscriptstyle{\perp}})\leq \omega(\frac{N}{q f})$ such $p$ that $S_f^p(s,\chi)>1$ in the last product, we have
\begin{align}
    S_f(s,\chi)\leq \frac{q f}{N}{4^{\omega(\frac{N}{q f})}} =:S_f(\chi).
\end{align}
Then (\ref{hardpart1}) follows from Proposition \ref{hardcore} and the fact
\begin{align*}
    \Big| \sum_{\mathfrak{a}}|\varphi_{\infty\mathfrak{a}}(s,\chi)|^2 \log \frac{N}{q f_{\mathfrak{a}}}  \Big| \leq \sum_{f \mid \frac{N}{q}} S_f(\chi) \log \frac{N}{q f}.
\end{align*}
We similarly have
\begin{align*}
    \sum_{\mathfrak{a}}|\varphi_{\infty\mathfrak{a}}(s,\chi)|^2 \sum_{p\mid \frac{N}{f_{\mathfrak{a}}}} \Big|\frac{\psi(p) \log p}{\psi(p) p^{2s-1}-1} \Big| \leq \sum_{f\mid \frac{N}{q}}\sum_{p\mid \frac{N}{q f}} \Big|S_f(s,\chi) \frac{\psi(p) \log p}{\psi(p) p^{2s-1}-1} \Big|.
\end{align*}
Noticing that $|S_f^p(s,\chi)\frac{1}{\psi(p) p^{2s-1}-1}|=\frac{|1-\psi(p) p^{1-2s}|}{|(1-\psi(p) p^{-2s})|^2}\leq \frac{2}{(1-p^{-1})^2}\leq 8$, we have
\begin{align*}
    \Big|\frac{S_f(s,\chi)}{\psi(p) p^{2s-1}-1} \Big|\ll S_f(\chi).
\end{align*}
Consequently,
\begin{align*}
    \sum_{f \mid \frac{N}{q}} \sum_{p\mid \frac{N}{q f}} \Big| S_f(s,\chi)  \frac{\psi(p) \log p}{\psi(p) p^{2s-1}-1} \Big| \ll \sum_{f \mid \frac{N}{q}}S_f(\chi)\sum_{p\mid \frac{N}{q f}}\log p \leq \sum_{f \mid \frac{N}{q}}S_f(\chi) \log \frac{N}{q f},
\end{align*}
and (\ref{hardpart2}) follows from Proposition \ref{hardcore}.
Equation (\ref{hardpart3}) results from (\ref{hardpart1}) and that $\sum_{\mathfrak{a}}|\varphi_{\infty\mathfrak{a}}(s)|^2 (\log f_{\mathfrak{a}} + \log \frac{N}{q f_{\mathfrak{a}}}) =\log \frac{N}{q} \sum_{\mathfrak{a}}|\varphi_{\infty\mathfrak{a}}(s)|^2 =\log \frac{N}{q}$ by Proposition \ref{Selberg}.
\end{proof}

\section{Integral renormalization}\label{SDIR}

\subsection{Equivalent definitions of integral regularizations}
We start by recalling Zagier's definition of integral regularizations on $Y_0(1)$.
Assume $F(z)$ is $SL_2(\mathbb{Z})$-invariant and satisfies
\begin{align}\label{moderately growth}
F(z) = \psi_{_F}(y) + O(y^{-P})
\end{align}
as $y \rightarrow \infty$ for all integers $P$, where $\psi_{_F} = \sum_{i=1}^{m} c_i y^{\alpha_i}$, with $c_i \in \mathbb{C}^*$, distinct $\alpha_i \in \mathbb{C}\backslash \{1\}$, $i=1,2,...,m$,  and $m=m(F) \geq 1$. When $m \neq 0$ and $\Re \alpha_i \geq 1$ for some $i$, $F$ is not integrable in the usual sense.  Nevertheless, $F$ is ``renormalizable" (in Zagier's terminology).  Write $R.N. (\int F d\mu)$, the \textit{renormalization} of $\int F d\mu$, defined by
\begin{itemize}
\item{$\int_{y <R} F d\mu + \int_{y \geq R} (F-\psi_{_F}) d\mu + \int^R y^{-2}\psi_{_F}(y) dy$}.
\end{itemize}
Here the first two integrals are performed over the standard fundamental domain $\mathcal{F}$ for $SL_2(\mathbb{Z})$, with their displayed additional restrictions, and the third is the ``\textit{anti-derivative}" with respect to $R$, i.e., a linear combination of $R$-powers without a nonzero constant term. Zagier's definition is independent of $R$, as we verify in the following subsection. Moreover, as we let $R \rightarrow \infty$, the second term tends to zero, giving an alternative definition:
\begin{itemize}
\item{$  \underset{R \rightarrow \infty}{\lim} \big(\int_{y <R} F d\mu - \int^R y^{-2}\psi_{_F}(y) dy \big)$}.
\end{itemize}
The third description is also called the \textit{regularization} of the integral $\int F d\mu$ by Michel and Venkatesh \cite{MV1}:
\begin{itemize}
\item $\int \big(F- \sum_{\substack{1 \leq i \leq m \\ \Re \thinspace\alpha_i\geq 1/2}} c_i E(z, \alpha_i) \big) d\mu$,
\end{itemize}
which is based on $R.N.\big( \int E(z,s) d\mu \big)=0$, a direct result of the following theorem.

\begin{mtheo}[Zagier \cite{Za}]
Assume $F$ is continuous, has Fourier expansion $\sum a_n(y)e(nx)$ and satisfies all above assumptions. Then $E(z,s) F(z)$ is also renormalizable for $\Re{s}$ large, and for any $R>1$ the following function
\begin{align}
\int_0^R a_0(y)y^{s-2} dy + \int_R^{\infty} (a_0(y) - \psi_{_F}(y)) y^{s-2} dy - \int^R  \psi_{_F} (y) y^{s-2}dy
\end{align}
has meromorphic continuation and equals $R.N. (\int E(z,s) F(z) d\mu)$.
\end{mtheo}

\subsection{Generalization of Zagier's result to arbitrary level}
By \cite[Proposition 2.4]{I2}, there exists a fundamental domain for $Y_0(N)$, whose vertices are $\Gamma_0(N)$-inequivalent cusps. Let $\mathcal{F}$ be such a fundamental domain.
For $R>1$, if we write $\mathcal{F}_{\mathfrak{a}}(R)$ to be the cuspidal zone, i.e., the image of the truncated strip $0<x<1, y>R$ under $\sigma_{\mathfrak{a}}$, and $\mathcal{F}(R) = \mathcal{F} \backslash \big( \bigsqcup_{\mathfrak{a}} \mathcal{F}_{\mathfrak{a}}(R)\big)$.

Assume $F(z) \in \mathcal{A}(Y_0(N))$ has Fourier expansion $\sum a_n(y)e(nx)$, and at each cusp $\mathfrak{a}$, there is $\psi_{\mathfrak{a}}=\sum_i c_{\mathfrak{a},i} y^{\alpha_{\mathfrak{a},i}}$, such that $i=1,2,...,m_{\mathfrak{a}}$ for some $m_{\mathfrak{a}} \geq 1$, and
\begin{align}\label{MG}
F(\sigma_{\mathfrak{a}} z)= \psi_{\mathfrak{a}}(y) + O(y^{-P}),
\end{align}
for all integers $P$ as $y \rightarrow \infty$, where  $c_{\mathfrak{a},i}\in \mathbb{C}\backslash\{0\}$ and $\alpha_{\mathfrak{a},i} \in \mathbb{C}\backslash\{1\}$. Then we call $F$ \textit{renormalizable}, because $\int F d\mu$ can be \textit{renormalized} as follows for all $R>1$:
\begin{align*}
R.N. \Big({\int_{\mathcal{F}} F(z) d\mu} \Big) := \int_{\mathcal{F}(R)} F d\mu + 
\underset{\mathfrak{a}}{\sum} \Big( \int_{\mathcal{F}_{\mathfrak{a}}(R)} \big( F(z)-\psi_{\mathfrak{a}}(\text{Im} \thinspace (\sigma_{\mathfrak{a}}^{-1} z)) \big) d\mu - \int^R \psi_{\mathfrak{a}} y^{-2}dy \Big).
\end{align*}

Again, the expression of the renormalized integral is independent of $R$: pick $1<R_1<R_2$, then the difference between the right hand sides of the equation evaluated at $R_2$ and $R_1$ is
\begin{multline*}
\int_{(\mathcal{F}(R_2)- \mathcal{F}(R_1))}  F d\mu - \underset{\mathfrak{a}}{\sum} \Big( \int_{\sigma_{\mathfrak{a}} (\mathcal{F}_{\infty}(R_1)- \mathcal{F}_{\infty}(R_2))} \big( F(z)-\psi_{\mathfrak{a}}(\Im \thinspace (\sigma_{\mathfrak{a}}^{-1} z)) \big) d\mu - \int_{R_2}^{R_1} \psi_{\mathfrak{a}}(y)y^{-2} dy \Big) \\
= \int_{\mathcal{F}_{\infty}(R_1)- \mathcal{F}_{\infty}(R_2)} \underset{\mathfrak{a}}{\sum}
\psi_{\mathfrak{a}}(y) d\mu - \underset{\mathfrak{a}}{\sum}  \int_{R_2}^{R_1}\psi_{\mathfrak{a}}(y)y^{-2} dy =0.
\end{multline*}
\begin{mrema}\label{usualintegral}
Just as in Zagier's level $1$ case, if the integrand is integrable already, the renormalized integral agrees with the usual integral.
\end{mrema}
Now suppose $F \in \mathcal{A}(Y_0(N), \overline{\chi})$ satisfies (\ref{MG}) and has Fourier expansion $\sum a^{\mathfrak{a}}_n(y)e(nx)$ at each $\mathfrak{a}$, with $\sum_{n\neq 0} |a_n^{\mathfrak{a}}(y)|= O(y^{-P})$ 
as $y\rightarrow\infty$ for all $P \geq 1$.
Define $R_{\mathfrak{a}}(F;s) := \int_0^{\infty}(a^{\mathfrak{a}}_0(y)-\psi_{\mathfrak{a}}(y))y^{s-2} dy$, which converges for $\Re s$ large by work of Dutta-Gupta \cite{DG}.

Hulse, Kuan, Lowry-Duda and Walker essentially generalized Zagier's theory to higher levels. Their original claim only concerns case $\chi$ being trivial, but it takes no extra effort to see that the same argument works for general central characters.
\begin{mtheo}\cite[Proposition A3]{HKL-DW}\label{GRN}
If $\Re s$ is sufficiently large, and $\mathfrak{a}\in \mathcal{C}_{\chi}(N)$, then
\begin{align*}
R.N. \big( \langle E_{\mathfrak{a}}(\cdot,s,\chi), \overline{F}(\cdot) \rangle_{_N} \big)= R_{\mathfrak{a}}(F;s).
\end{align*}
\end{mtheo}

Consequently, the renormalized integral of a single Eisenstein series, attached to any cusp, vanishes, which justifies the third definition in Zagier's work, as well as our generalization:
\begin{align*}
R.N. \Big( \int F d\mu \Big) = \int \Big( F- \sum_{\mathfrak{a}} \sum_{\text{Re}\thinspace\alpha_{\mathfrak{a},i} \geq 1/2} E_{\mathfrak{a}}(z, \alpha_{\mathfrak{a},i}) \Big).
\end{align*}
We also call this the \textit{regularization} of $\langle F, 1 \rangle_{_N}$ and write it $\langle F, 1 \rangle_{_N}^{\text{reg}}$.
\begin{mcoro}\label{RNzero}
For any $\mathfrak{a}$ and $\mathfrak{b}$ singular for $\chi$ and $s_1, s_2 \in \mathbb{C}\backslash\{0,1\}$, we have
\begin{align*}
\langle E_{\mathfrak{a}}(\cdot, s_1, \chi), E_{\mathfrak{b}}(\cdot, s_2, \chi) \rangle_{_N}^{\mathrm{reg}} = 0.
\end{align*}
\end{mcoro}
\begin{mrema}
Note the difference between $\langle \cdot, \cdot \rangle^{\mathrm{reg}}_{_N}$ above and $\langle \cdot, \cdot \rangle^{\text{Eis}}_{_N}$ from Lemma \ref{renormalization}.
\end{mrema}

\section{Spectral decomposition}\label{spec}
Here we take the notation in \cite{I2} of $\mathcal{B}_{\delta}(Y_0(N))$ with $\delta\geq 0$, which stands for the space of smooth automorphic functions $f$ on $Y_0(N)$, satisfying
\begin{align*}
f(\sigma_{\mathfrak{a}} z) \ll y^{\delta} \quad \text{ as} \quad y\rightarrow \infty,
\end{align*}
for all $\mathfrak{a}\in \mathcal{C}(N)$. We note that for $\delta < \frac{1}{2}$, $\mathcal{B}_{\delta}(Y_0(N)) \subset L^2(Y_0(N))$.

\subsection{Classical theory}\label{cos}
For $F\in\mathcal{B}_\delta(Y_0(N))$ with $\delta < 1/2$, we have spectral decomposition:
\begin{align*}
    F(z) = \frac{\langle F,1 \rangle_{_N}}{\langle 1,1 \rangle_{_N}} + \sum_{u\in\mathcal{O}(N)}\langle F,u \rangle_{_N} u(z) + \frac{1}{4\pi}\sum_{\mathfrak{a}\in\mathcal{C}(N)} \int_{-\infty}^{\infty} \langle F,E_{\mathfrak{a}}(\cdot, \tfrac{1}{2}+it) \rangle_{_N} E_{\mathfrak{a}}(z, \tfrac{1}{2}+it) dt.
\end{align*}
\begin{mrema}
In our work, the choice of $E_{\mathfrak{a}}$ as an orthogonal basis in the spectral decomposition is convenient for computations with the main terms, but not for the error terms.
\end{mrema}
%Here $\frac{1}{4\pi}$ serves as the formal $L^2$-renormalization factor. This choice of formal orthogonal basis is concise and is acceptable in the main term calculation. However, for simplicity of the error term estimation, we prefer something else.

\subsection{Regularization for spectral decomposition}
To apply the spectral decomposition, we need to regularize $|E|^2$. See \cite[Sections 4.3--4.4]{MV1} for more about the general theory.

\begin{mprop}\label{id}
For $E = E_{\infty}(z,\frac{1}{2}+iT,\chi)$ as in Theorem \ref{main}, we have $|E|^2 - \mathcal{E} \in \mathcal{B}_{\varepsilon}(Y_0(N))$ for arbitrarily small $\varepsilon > 0$ with
\begin{multline*}
\mathcal{E} := 2\Re\Big( \varphi_{\infty\infty}(\tfrac{1}{2}+iT,\chi) E_{\infty}(z,1-2iT) \Big) \\
+ \lim_{\beta\rightarrow 0^+} \Big(  E_{\infty}(z,1+\beta) + \sum_{\mathfrak{a}\in \mathcal{C}_{\chi}(N)} \varphi_{\infty\mathfrak{a}}(\tfrac{1}{2}+iT,\chi)\varphi_{\infty\mathfrak{a}}(\tfrac{1}{2}+\beta-iT,\overline{\chi}) E_{\mathfrak{a}}(z,1-\beta) \Big).
\end{multline*}
\end{mprop}
\begin{mrema}
We note that as long as $T\neq 0$, $\mathcal{E}$ is well-defined as an element in $\mathcal{B}_{\varepsilon}(Y_0(N))$. See \cite[Section 2]{Wu} for an extension for the case of $T=0$.
\end{mrema}
\begin{proof}
This is done by comparing $\psi_{_{F_{\beta}}}$ (see (\ref{moderately growth}) for definition) with $\psi_{_{\mathcal{E}_{\beta}}}$ for \begin{align*}
F_{\beta}(z,T) &= E_{\infty}(z,\tfrac{1}{2}+iT,\chi) E_{\infty}(z,\tfrac{1}{2}+\beta-iT,\overline{\chi})  \text{ and }\\ 
\mathcal{E}_{\beta}(z,T) &= \varphi_{\infty\infty}(\tfrac{1}{2}+iT,\chi) E_{\infty}(z,1+\beta-2iT) + \varphi_{\infty\infty}(\tfrac{1}{2}+\beta -iT,\overline{\chi}) E_{\infty}(z,1-\beta+2iT) \\
 & \quad\quad\quad\quad\quad + E_{\infty}(z,1+\beta) + \sum_{\mathfrak{a}} \varphi_{\infty\mathfrak{a}}(\tfrac{1}{2}+iT,\chi)\varphi_{\infty\mathfrak{a}}(\tfrac{1}{2}+\beta-iT,\overline{\chi}) E_{\mathfrak{a}}(z,1-\beta).
\end{align*}
The constant terms in the Fourier expansion of $E_{\infty}$ can be calculated via (\ref{expansion}) and (\ref{specialf=N}), and that of $E|_{\sigma_{\mathfrak{a}}}$ is computable with Proposition \ref{scatteringmatrix}. Now that $\psi_{F_{\beta}}$ and $\psi_{\mathcal{E}_{\beta}}$ agree for all sufficiently small $\beta >0$, their difference lies in $ \mathcal{B}_{\varepsilon}(Y_0(N))$, for all $\varepsilon > \beta$.
\end{proof}

\subsection{Regularized spectral decomposition in a new choice of orthonormal basis}
Define 
\begin{align}\label{ONS}
    \mathcal{O}_j(M) := \Big\{ u_j^{\scriptscriptstyle{<\ell>}}(z)=\sum_{d\mid \ell}\xi_{\ell}(d)u_j|_d \quad \Big| \quad u_j \in \mathcal{H}_{it_j}(M_1), \ell \mid M_2, M=M_1 M_2 \Big\},
\end{align} 
where $\mathcal{H}_{it_j}(M_1)$ stands for the set of $L^2(Y_0(M))$-normalized Hecke--Maass newforms of level $M_1$ and spectral parameter $t_j$, and 
$\xi_{\ell}(d)$ are certain coefficients satisfying the bound
\begin{align}\label{BM}
\xi_{\ell}(d) \ll \ell^{\varepsilon} (\ell/d)^{\theta-\frac{1}{2}},
\end{align}
as is described in \cite[(5.6)]{BM}.
Here each $u_j$ can be written as $\rho_j u_j^*$, where 
\begin{align}\label{uj}
    u_j^*(z)= \sqrt{y}\sum_{n\neq 0} \lambda_j(n) K_{it_j}(2\pi |n| y)e(nx),
\end{align}
stands for the Hecke-normalized cusp form, and 
\begin{align}\label{rho2}
\rho_j= \norm{u_j^*}_2^{\scriptscriptstyle{-1}}=O(M^{-\frac{1}{2}+\varepsilon}e^{\frac{\pi |t_j|}{2}}).
\end{align}
Blomer and Mili\'cevi\'c \footnote{See https://www.uni-math.gwdg.de/blomer/corrections.pdf for corrections of \cite{BM}.} showed that $\mathcal{O}_j(M)$ is an orthonormal basis of the space of cusp forms of spectral parameter $t_j$. Consequently, $\mathcal{O}(M):= \bigsqcup_{j =1}^{\infty} \mathcal{O}_j(M)$ makes an orthonormal basis of Maass cusp forms of level $M$.

Parallelly, as explained in \cite[Section 8.3]{Y2},
\begin{align}\label{SNS}
    \mathcal{O}^{\text{Eis}}_t(M) := \Big\{ E^{\scriptscriptstyle{<\ell>}}_{\eta,\eta}(z,\tfrac{1}{2}+it) = \frac{\sum_{d\mid \ell}\xi^{\mathrm{Eis}}_{\ell}(d) E_{\eta,\eta}(dz,\tfrac{1}{2}+it)}{\norm{E^{\scriptscriptstyle{(M)}}_{\eta,\eta}}_2^{\text{Eis}}} \medspace \Big|  \medspace \eta \mymod{r}, r^2 \ell \mid M \Big\}
\end{align}
forms a formal orthonormal basis, where $\xi^{\mathrm{Eis}}_{\ell}(d)$ also satisfy the same bound as $\xi_{\ell}(d)$ in \eqref{BM}, being obtained in the same way via substitutions of the Hecke eigenvalues. 
By Lemma \ref{renormalization},
\begin{align*}
    \norm{E^{\scriptscriptstyle{(M)}}_{\eta,\eta}}_2^{\mathrm{Eis}}:=\sqrt{\langle E_{\eta,\eta}^{\scriptscriptstyle{(M)}}, E_{\eta,\eta}^{\scriptscriptstyle{(M)}} \rangle_{_M}^{\mathrm{Eis}}}
    = \sqrt{4\pi M}\prod_{p\mid r}(1-p^{-1})^{\frac{1}{2}} \prod_{p\mid M_r^{\scriptscriptstyle{\perp}}}(1+p^{-1})^{\frac{1}{2}} = M^{\frac{1}{2}+o(1)}.
\end{align*}

From the definition of renormalized integral and Corollary \ref{RNzero}, we have $\langle |E|^2 - \mathcal{E}, 1\rangle_{_N}=0$.
Since $\langle \mathcal{E}, u \rangle_{_N} =0$, applying the Plancherel formula to $ \langle |E|^2 - \mathcal{E}, \phi \rangle_{_N}$ yields
\begin{align}\label{ET}
\langle |E|^2 - \mathcal{E}, \phi \rangle_{_N}  =\sum_{u\in \mathcal{O}(M)} \langle |E|^2, u \rangle_{_N} \langle u, \phi \rangle_{_M} 
+ \int_{-\infty}^{\infty} \sum_{E_t\in \mathcal{O}^{\text{Eis}}_t(M)} \langle |E|^2 , E_t\rangle_{_N}^{\text{reg}} \langle E_t, \phi \rangle_{_M} dt.
\end{align}

 Consequently we can take (\ref{ONS}) and (\ref{SNS}) back to (\ref{ET}), and obtain
\begin{multline}\label{reget}
\langle |E|^2 - \mathcal{E}, \phi \rangle_{_N} 
= 
\sum_{j\geq 1} \sum_{M_1 M_2 = M} \sum_{u_j\in \mathcal{H}_{it_j}(M_1)} \sum_{\ell\mid M_2} \langle |E|^2 , u_j^{\scriptscriptstyle{<\ell>}} \rangle_{_N} \langle u_j^{\scriptscriptstyle{<\ell>}}, \phi \rangle_{_M} 
\\
+ \int_{-\infty}^{\infty}\sum_{r^2 L= M} \medspace\medspace \sideset{}{^*}\sum_{\eta \shortmod{r}} \sum_{\ell \mid L}  \langle |E|^2 , E^{\scriptscriptstyle{<\ell>}}_{\eta,\eta}(\cdot, \tfrac{1}{2}+it)\rangle^{\text{reg}}_{_N} \langle E^{\scriptscriptstyle{<\ell>}}_{\eta,\eta}(\cdot, \tfrac{1}{2}+it), \phi \rangle_{_M} dt,
\end{multline}
where the asterisked sum is over all primitive Dirichlet characters $\mymod{r}$.
We estimate the terms in (\ref{ET}), or equivalently (\ref{reget}), and $\langle \mathcal{E}, \phi \rangle_{_N}$ in the following sections.

\section{Error term estimation}\label{ETE}

\subsection{Calculation with Fourier coefficients}

\begin{mlemma}\label{AB}
Let $\Re s$ be sufficiently large. Suppose $f \in \mathcal{A}(Y_0(N), \chi)$ and $g \in \mathcal{A}(Y_0(N))$ have Fourier expansions
\begin{align*}
f(z) &= a_0(y) + \sqrt{y}\sum_{n\neq 0} \lambda_f(n) a(ny) e(nx) \\ g(z) &= \sqrt{y}\sum_{n\neq 0} \lambda_g(n) b(ny) e(nx),
\end{align*}
where $\lambda_f$ and $\lambda_g$ are multiplicative and $\lambda_*(-n)=\lambda_*(-1)\lambda_*(n)$ for $*=f$ or $g$. Then we have
\begin{align*}
\langle E^{\scriptscriptstyle{(N)}}_{\infty}(\cdot, s, \chi),  f \cdot g \rangle_{_N} = (\overline{\lambda_f}(-1) + \overline{\lambda_g}(-1))  h(s) \sum_{n\geq 1} n^{-s} \overline{\lambda_f}(n) \overline{\lambda_g}(n),
\end{align*}
where $h(s)=\int_0^{\infty}y^{s-1} \overline{a(y) b(y)} dy$.
\end{mlemma}

\begin{proof}
This is easy by unfolding and integration on $x$.
\end{proof}

\begin{mcoro}\label{unfolding}
With the same assumptions as Lemma \ref{AB}, if we further have $f|_A\in \mathcal{A}(Y_0(N), \chi)$ and $g|_B \in \mathcal{A}(Y_0(N))$ for some $A,B \mid N$, then
\begin{align*}
\langle E^{\scriptscriptstyle{(N)}}_{\infty}(\cdot, s, \chi),  f|_A \cdot g|_B \rangle_{_N} = (\overline{\lambda_f}(-1) + \overline{\lambda_g}(-1))  h(s) Z_{A,B}(s),
\end{align*}
with
\begin{align*}
Z_{A,B}(s)= \frac{\sqrt{AB}}{[A,B]^s} \sum_{n\geq 1} n^{-s} \overline{\lambda_f} \Big(\frac{[A,B]}{A}n \Big) \overline{\lambda_g} \Big(\frac{[A,B]}{B}n \Big).
\end{align*}
\end{mcoro}

\subsection{Cuspidal contribution}
The following corollary is a special case of Corollary \ref{unfolding} with (\ref{uj}) and (\ref{specialf=N}).
\begin{mcoro}
For all $A\mid \frac{N}{q}$ and $B\mid N$, we have
\begin{align*}
    \langle E^{(N)}_{\infty}(\cdot,\tfrac{1}{2}+iT, \chi), E_{1,\overline{\psi}}|_{A} \cdot u_j|_B \rangle_{_N} = F_T(t_j) Z_{A,B}(\tfrac{1}{2}+iT,\psi, u_j),
\end{align*}
where
\begin{align*}
Z_{A,B}(\tfrac{1}{2}+iT,\psi, u_j) &= \frac{\sqrt{AB}}{[A,B]^{\frac{1}{2}+iT}} \sum_{n\geq 1}\frac{\overline{\lambda_{1,\overline{\psi}}}(\frac{[A,B]}{A} n) \lambda_j(\frac{[A,B]}{B} n)}{n^{\frac{1}{2}+iT}}, \text{ and}\\
   F_T(t_j) = \overline{\rho_{1,\overline{\psi}}(\tfrac{1}{2}+iT)\rho_j} (\overline{\lambda_{1,\overline{\psi}}} (-&1) +\lambda_j(-1))\int_0^{\infty} y^{-\frac{1}{2}+iT} K_{iT}(2 \pi y) K_{it_j} (2 \pi y) dy.
\end{align*}
\end{mcoro}

From (\ref{rho1}), (\ref{rho2}), and \cite[(6.576.4)]{GR}, we see $F_T(t_j) \ll N^{\varepsilon}M^{-\frac{1}{2}}e^{H_T(t_j)} P(t_j,T)$ for some polynomial $P(x,y)$, where
\begin{align}\label{HT}
H_T(t_j)=\begin{cases} 0 &\text{ if } |t_j|\leq 2|T|, \\ \frac{\pi}{2} (2|T|-|t_j|) &\text{ if } |t_j|>2|T|. \end{cases}
\end{align}
As for $Z_{A,B}(\frac{1}{2}+iT,\psi, u_j)$, we can rewrite the Dirichlet series as an Euler product
\begin{align*}
\frac{\sqrt{AB}}{[A,B]^{\frac{1}{2}+iT}}   \prod_p \Big( \sum_{n\geq 0} \frac{\overline{\lambda_{1,\overline{\psi}}}(p^{n+\nu_p(\frac{[A,B]}{A})}) \overline{\lambda_j}(p^{n+\nu_p(\frac{[A,B]}{B})})}{p^{n(\frac{1}{2}+iT)}} \Big) = F_j(A,B)  \sum_{n\geq 1}\frac{\overline{\lambda_{1,\overline{\psi}}}(n) \lambda_j(n)}{n^{\frac{1}{2}+iT}},
\end{align*}
where $F_j(A,B)$ is a finite Euler product over prime divisors of $[A,B]$. Inserting the bounds from Remark \ref{lambda1} and Convention \ref{conv1}, we have $F_j(A,B) = O(N^{\varepsilon}(A,B)^{\frac{1}{2}} (A_M^{\scriptscriptstyle{\perp}})^{\theta})$. Applying the Rankin--Selberg method (see e.g. \cite[(13.1)]{I1}), we have
\begin{align*}
    \sum_{n\geq 1}\frac{\overline{\lambda_{1,\overline{\psi}}}(n) \lambda_j(n)}{n^{\frac{1}{2}+iT}} = \frac{L(\frac{1}{2}, u_j)L(\frac{1}{2}+2iT, u_j \otimes \overline{\psi})}{L(1+ 2iT,\overline{\psi}\cdot \chi_{_0,_M})}.
\end{align*}
Recalling equation (\ref{expansion}) and the fact $|L(1+ 2iT,\overline{\psi})|\gg_T q^{-\varepsilon}$, we have the following lemma.

\begin{mlemma}\label{thefirstcalculationidid}
Keeping above notations and $s=\tfrac{1}{2}+iT$, we have for all $d\mid M$
\begin{align*}
    \langle |E_{\infty}(\cdot,s,\chi)|^2, u_j|_d \rangle_{_N} \ll_{_T}e^{H_T(t_j)}  N^{-\frac{1}{2}+\varepsilon} M^{-\frac{1}{2}} (\tfrac{N}{q}, d )^{\frac{1}{2}} (\tfrac{N}{q} )^{\theta} |L(\tfrac{1}{2},u_j) L(\tfrac{1}{2}+2iT,u_j\otimes \overline{\psi})|.
\end{align*}
\end{mlemma}
Notice Lemma \ref{thefirstcalculationidid} implies Proposition \ref{et}. Now we can estimate the first part of (\ref{ET}).
\begin{mprop}\label{ET1}
Keeping all notations in Theorems \ref{secondary} and \ref{main}, we have
\begin{align*}
   \sum_{u\in \mathcal{O}(M)} \langle |E_{\infty}(\cdot, \tfrac{1}{2}+iT, \chi)|^2, u \rangle_{_N} \langle u, \phi \rangle_{_M}
    \ll_{_T} N^{-\frac{1}{2}+\varepsilon} (\tfrac{N}{q})^{\theta} M^{\frac{1}{2}} q^{\frac{3}{8}}  \norm{\phi}_2.
\end{align*}
\end{mprop}
Before proving Proposition \ref{ET1}, we claim a lemma.
\begin{mlemma}\label{LY}
We have
\begin{align*}
    \sum_{t_j \leq 2|T| + 2\log N} \sum_{u_j \in \mathcal{H}_{it_j}(M_1)} |L(\tfrac{1}{2}, u_j)|^2 \ll_{T,\varepsilon} N^{\varepsilon}M_1.
\end{align*}
\end{mlemma}
The proof follows from the spectral large sieve inequality, so we omit it. See Motohashi \cite[(3.4.4)]{Mo} for an example on the case $M=1$.

\begin{mrema}
A bound of the same 
quality actually holds for the fourth moment of central values of these $L$-functions, which follows from the spectral large sieve for $\Gamma_0(M)$ developed by Deshouillers and Iwaniec \cite{DI}. Motohashi \cite[Theorem 3.4]{Mo} shows this for the case $M=1$.
\end{mrema}
\begin{proof}[Proof of Proposition \ref{ET1}]
By (\ref{ONS}), (\ref{reget}) and Cauchy--Schwarz, we have
\begin{multline*}
    \sum_{u\in \mathcal{O}(M)} |\langle |E_{\infty}|^2, u \rangle_{_N} \langle u, \phi \rangle_{_M}| =  \sum_{j\geq 1}\sum_{u_j\in \mathcal{O}_j(M)} |\langle |E_{\infty}|^2, u_j \rangle_{_N} \langle u_j, \phi \rangle_{_M}| \\
    \leq \Big( \sum_{j\geq 1} \sum_{u_j\in \mathcal{O}_j(M)}  |\langle |E_{\infty}|^2, u_j \rangle_{_N}|^2 \Big)^{\frac{1}{2}} 
    \Big( \sum_{j\geq 1} \sum_{u_j\in \mathcal{O}_j(M)}  |\langle u_j, \phi \rangle_{_M} |^2 \Big)^{\frac{1}{2}}.
\end{multline*}
Observe that by Bessel's inequality,
\begin{align*}
    \sum_{j \geq 1} \sum_{u_j\in \mathcal{O}_j(M)}  |\langle u_j, \phi \rangle_{_M} |^2  \leq \norm{\phi}_2^2.
\end{align*}
As for the other factor, we recall (\ref{ONS}) and (\ref{BM}), and apply Cauchy--Schwarz  again to see
\begin{align*}
    |\langle |E_{\infty}|^2, u_j^{\scriptscriptstyle{<\ell>}} \rangle_{_N}| &\leq \Big( \sum_{d\mid \ell} |\xi_d^{\scriptscriptstyle{<\ell>}}|^2 \Big)^{\frac{1}{2}} \Big(\sum_{d\mid \ell} |\langle |E_{\infty}|^2, u_j|_d \rangle_{_N}|^2 \Big)^{\frac{1}{2}} \ll \ell^{\varepsilon} \max_{d \mid \ell} |\langle |E_{\infty}|^2, u_j|_d \rangle_{_N}| \\
    &\ll_{\varepsilon} N^{-\frac{1}{2}+\varepsilon} M^{-\frac{1}{2}} e^{H_T(t_j)}\big(\tfrac{N}{q}, \ell \big)^{\frac{1}{2}}\big(\tfrac{N}{q} \big)^{\theta} \Big|L \big(\tfrac{1}{2},u_j \big) L \big(\tfrac{1}{2}+2iT,u_j\otimes \overline{\psi} \big) \Big|,
\end{align*}
where $\xi_d^{\scriptscriptstyle{<\ell>}}$ is defined in (\ref{ONS}).Because of the factor $e^{H_T(t_j)}$ (see (\ref{HT}) for its magnitude), 
%in each $\langle|E|^2, u_j\rangle_{_N}$, 
%the summation over $j$ of $|\langle |E_{\infty}|^2, u_j \rangle_{_N}|^2$ can be truncated to these $j$ with $|t_j|\leq 2|T|+2\log M$, with tail
we may truncate the sum at $|t_j|\leq 2|T|+2\log N$, with a very small error term. 
%[note: let's just truncate at $\log{N}$ rather than $\log{M}$!]

% \begin{align*}
%     N^{-1+\varepsilon}\Big(\frac{N}{q}\Big)^{2\theta} M^{-1} \sum_{|t_j| > 2|T|+2\log M} e^{2H_T(t_j)} \sum_{u_j^{<\ell>}\in \mathcal{O}_j(M)} \Big(\frac{N}{q},\ell\Big) \Big|L \Big(\frac{1}{2},u_j \Big) L \Big(\frac{1}{2}+2iT,u_j\otimes \overline{\psi} \Big) \Big|^2 ,\\
%     \ll N^{-1+\varepsilon}\Big(\frac{N}{q}\Big)^{2\theta} M^{-5} \Big(\frac{N}{q},M\Big) \sum_{|t_j| > 2|T|} q^{\frac{3}{4}} M^2 e^{4|T|-2|t_j|} (|2T|+1)^3 \ll_{_T} N^{-1+\varepsilon}\Big(\frac{N}{q}\Big)^{2\theta} M^{-2} q^{\frac{3}{4}},
% \end{align*}
% where the estimation of the $L$-functions uses Theorem \ref{BH}, and Weyl's bound on the dimension of $\mathcal{O}_j(M)$ is applied. 
Furthermore, for all $|t_j|\leq 2|T|+2\log N$, we have
\begin{multline*}
    \sum_{l\mid M_2} |\langle |E_{\infty}|^2, u_j^{\scriptscriptstyle{<\ell>}}\rangle_{_N}|^2 \ll_{\varepsilon} N^{-1+\varepsilon} \frac{\sum_{l\mid M_2} (\frac{N}{q},\ell)}{M} \big(\tfrac{N}{q}  \big)^{2\theta} |L (\tfrac{1}{2},u_j) L (\tfrac{1}{2}+2iT,u_j\otimes \overline{\psi} ) |^2 \\
    = N^{-1+\varepsilon}M^{-1} (\tfrac{N}{q} )^{2\theta} (\tfrac{N}{q},M_2) |L(\tfrac{1}{2},u_j ) L(\tfrac{1}{2}+2iT,u_j\otimes \overline{\psi} ) |^2,
\end{multline*}
and by Theorem \ref{BH} and Lemma \ref{LY}, we have
\begin{multline*}
    \sum_{|t_j| \leq 2|T|+2\log N}\sum_{u_j \in \mathcal{H}_{it_j}(M_1)}\sum_{\ell \mid M_2} |\langle |E_{\infty}|^2, u_j^{\scriptscriptstyle{<\ell>}}\rangle_{_N}|^2 \\ \ll_{_T} N^{-1+\varepsilon} M^{-1} (\tfrac{N}{q})^{2\theta} (\tfrac{N}{q},M_2) M_1 \max\{M_1^{\frac{1}{2}}q^{\frac{3}{4}}, M_1(M_1,q)^{\frac{1}{2}}q^{\frac{1}{2}}\}.
\end{multline*}
In the summation over $M_1 M_2 =M$, the term with $M=M_1$ and $M_2=1$ dominates, so
\begin{align*}
    \Big( \sum_{|t_j| \leq 2|T|+2\log N} \sum_{u_j\in \mathcal{O}_j(M)} |\langle |E_{\infty}|^2, u_j \rangle_{_N}|^2 \Big)^{\frac{1}{2}} \ll_{_T} N^{-\frac{1}{2}+\varepsilon} (\tfrac{N}{q})^{\theta} \max\{ M^{\frac{1}{4}} q^{\frac{3}{8}}, M^{\frac{1}{2}} (M,q)^{\frac{1}{4}} q^{\frac{1}{4}} \}. \quad\qedhere
\end{align*}
\end{proof}

\begin{mrema}\label{LY'}
Following the same line as Lemma \ref{LY} we can similarly have
\begin{align*}
   \sideset{}{^*}\sum_{\eta \shortmod{r}}   \int_{-2|T|-2\log N}^{2|T|+2\log N} |L(\tfrac{1}{2},E_{\eta,\eta}(\cdot, \tfrac{1}{2}+it))|^2 dt \ll_{_T} N^{\varepsilon} r.
\end{align*}
\end{mrema}

\subsection{Eisenstein contribution}
Now we estimate the second part in (\ref{ET}). It is not hard to see we have made every piece correspond well with their cusp form counterpart in the rewritten formula (\ref{reget}), and that is why we choose $\mathcal{O}_t^{\text{Eis}}(M)$ to be the orthonormal basis.

\begin{mlemma}\label{thefirstcalculationidid'}
Keeping all notations as in (\ref{reget}), we have
\begin{align*}
\langle |E_{\infty}(\cdot,s,\chi)|^2, E_{\eta,\eta}(d\cdot, \tfrac{1}{2}+it)\rangle_{_N}^{\mathrm{reg}} \ll_{_T} e^{H_T(t)} N^{-\frac{1}{2}+\varepsilon} (\tfrac{N}{q},d)^{\frac{1}{2}} |L(\tfrac{1}{2}, E_{{\eta},{\eta}}) L(\tfrac{1}{2}+2iT, E_{{\eta},{\eta}}\otimes\overline{\psi} )|,
\end{align*}
$H_T(\cdot)$ being the same as in (\ref{HT}). Here $s=\tfrac{1}{2}+iT$ and $E_{\eta,\eta}$ in the $L$-functions depends on $t$.
\end{mlemma}
%The convenience begins with the easiness of the acquisition of the following lemma in correspondence with Lemma \ref{thefirstcalculationidid}. 
The proof is almost the same as that of Lemma \ref{thefirstcalculationidid}, so we omit the details.

\begin{mprop}\label{ET2}
Keeping all notations from Theorems \ref{secondary} and \ref{main}, we have
\begin{align*}
\int_{-\infty}^{\infty} \sum_{E_t \in \mathcal{O}^{\text{Eis}}_t(M)} \langle |E|^2 , E_t \rangle_{_N}^{\text{reg}} \langle E_t, \phi \rangle_{_M} dt \ll_{_T} N^{-\frac{1}{2}+\varepsilon} q^{\frac{3}{8}} M^{\frac{1}{2}} \norm{\phi}_{_2}.
\end{align*}
\end{mprop}

\begin{proof}[Sketch of proof]
After Lemma \ref{thefirstcalculationidid'}, the calculation can be reduced to some multiple of
\begin{align*}
 \sideset{}{^*}\sum_{\eta \shortmod{r}}  \int_{-2|T|-2\log N}^{2|T|+2\log N} |L(\tfrac{1}{2}, E_{{\eta},{\eta}})L(\tfrac{1}{2}+2iT, E_{{\eta},{\eta}}\otimes \overline{\psi})|^2 dt,
\end{align*}
with a similarly negligible tail. Then we can just perform the same procedure of proving Proposition \ref{ET1}, except for taking the Burgess bound for $|L(\tfrac{1}{2}+2iT, E_{{\eta},{\eta}}\otimes \overline{\psi})|$ instead of Theorem \ref{BH}, and putting the equation in Remark \ref{LY'} in place of Lemma \ref{LY}. 
\end{proof}
\begin{mrema}
One may improve this error term by using the Weyl bound for Dirichlet $L$-functions \cite{PY1, PY2}, reducing the exponent of $q$ on the right hand side from $3/8$ to $1/3$.  However, this does not improve the overall error term which is limited by the cusp form contribution.
\end{mrema}

\section{Main term estimation}\label{MTE}
The main goal of this section is to prove \eqref{eq:mainthemMainTerm} and \eqref{eq:alphaphibound}, which are the main term aspects of Theorem \ref{main}. Throughout this section we adopt all notations in previous sections.
\subsection{Preparation}
Recall $W^1_N(\mathfrak{a})$ is the width of $\mathfrak{a}$ (see Section \ref{AW} for definition).
\subsubsection{Weighted average}
\begin{mlemma}\label{weightedavg}
For $s=\frac{1}{2}+iT$, we have
\begin{multline*}
-\sum_{\mathfrak{a}\in\mathcal{C}_{\chi}(N)} |\varphi_{\infty\mathfrak{a}}(s,\chi)|^2 \Big(\frac{\varphi'_{\infty\mathfrak{a}}(\overline{s},\overline{\chi})}{\varphi_{\infty\mathfrak{a}}(\overline{s},\overline{\chi})} + \log W^1_N(\mathfrak{a}) \Big) = 2 \log N +4 \Re \frac{L'(1+2iT,\overline{\psi})}{L(1+2iT, \overline{\psi})} \\
+ O_T(1) + O\big( \big(\log\log \big(\tfrac{N}{q} +2 \big) \big)^5 \big).
\end{multline*}
\end{mlemma}
\begin{proof}
According to Corollary \ref{varphi}, for $\mathfrak{a}=\frac{u}{f}\in \mathcal{C}_{\chi}(N)$ with $f\mid \frac{N}{q}$, we have
\begin{multline*}
   - \frac{\varphi'_{\infty\mathfrak{a}}(\tfrac{1}{2}-iT, \overline{\chi})}{\varphi_{\infty\mathfrak{a}}(\tfrac{1}{2}-iT, \overline{\chi})} = -(\log \varphi_{\infty\mathfrak{a}}(\tfrac{1}{2}-iT, \overline{\chi}))' = \log \frac{fN}{(f,\frac{N}{f})}\\
     +4 \Re \frac{\Lambda'(1+2iT,\overline{\psi})}{\Lambda(1+2iT, \overline{\psi})} +2\sum_{p\mid N}\frac{\psi(p)p^{-1+2iT}\log p}{1-\psi(p)p^{-1+2iT}} - 2\sum_{p\mid \frac{N}{f}}\frac{\psi(p)p^{2iT}\log p}{1- \psi(p)p^{2iT}},
\end{multline*}
where $\Lambda$ is the completed $L$-function. Moreover, by Lemma \ref{width} and Proposition \ref{Selberg}, we have
\begin{multline*}
    \sum_{\mathfrak{a}\in\mathcal{C}_{\chi}(N)} - |\varphi_{\infty\mathfrak{a}}(\tfrac{1}{2}+iT,\chi)|^2 \Big( \frac{\varphi'_{\infty\mathfrak{a}}(\tfrac{1}{2}-iT, \overline{\chi})}{\varphi_{\infty\mathfrak{a}}(\tfrac{1}{2}-iT, \overline{\chi})} + \log W^1_N(\mathfrak{a}) \Big) = \sum_{\mathfrak{a}\in\mathcal{C}_{\chi}(N)} |\varphi_{\infty\mathfrak{a}}(\tfrac{1}{2}+iT,\chi)|^2  \\
\cdot \Big( 2 \log f +4 \Re \frac{\Lambda'(1+2iT,\overline{\psi})}{\Lambda(1+2iT, \overline{\psi})} +2\sum_{p\mid N}\frac{\psi(p)p^{-1+2iT}\log p}{1- \psi(p)p^{-1+2iT}} - 2\sum_{p\mid \frac{N}{f}}\frac{\psi(p)p^{2iT}\log p}{1- \psi(p)p^{2iT}} \Big).
\end{multline*}
Recalling Corollaries \ref{easypart} and \ref{hardpart}, we arrive at the lemma.
\end{proof}

\subsubsection{Traced Eisenstein series}
Applying the trace operator $\Tr^N_M$ (see the definition in (\ref{trace})) to $\mathcal{E}$, we have (see \cite[Lemma 12]{AL})
\begin{align*}
\langle \mathcal{E}, \phi \rangle_{_N} = \langle \Tr^N_M \mathcal{E}, \phi \rangle_{_M}.
\end{align*}
To calculate further with this, we need to identify $\Tr^N_M \mathcal{E}$. By Lemma \ref{relativetrace} and Proposition \ref{id}, we have for all $T\neq 0$
\begin{multline}\label{7dot1}
\Tr^N_M \mathcal{E} = 2\Re \Big( \varphi^{\scriptscriptstyle{(N)}}_{\infty\infty}(\tfrac{1}{2}+iT,\chi) E_{\infty}^{\scriptscriptstyle{(M)}}(z,1-2iT) \Big) + \lim_{\beta\rightarrow 0^+} \Big( E_{\infty}^{\scriptscriptstyle{(M)}}(z,1+\beta) + \\ 
 \sum_{\mathfrak{a}\in \mathcal{C}_{\chi}(N)}\varphi^{\scriptscriptstyle{(N)}}_{\infty\mathfrak{a}}(\tfrac{1}{2}+iT,\chi) \varphi^{\scriptscriptstyle{(N)}}_{\infty\mathfrak{a}}(\tfrac{1}{2}+\beta-iT, \overline{\chi}) (W^M_N(\mathfrak{a}))^{\beta} E_{\mathfrak{a}}^{\scriptscriptstyle{(M)}}(z,1-\beta) \Big).
\end{multline}

It is still necessary to simplify (\ref{7dot1}) further.
\begin{mprop}\label{altE}
When $T\neq 0$, we have
\begin{align*}
    \Tr^N_M \mathcal{E} = c_0 + \sum_{g\mid M}c_g G|_g + 
\sum_{g|M} c_g' E(\cdot, 1 + 2iT)|_g,  
%    2\Re \Big( \varphi^{\scriptscriptstyle{(N)}}_{\infty\infty}(\tfrac{1}{2}+iT,\chi) E_{\infty}^{\scriptscriptstyle{(M)}}(z,1-2iT) \Big),
\end{align*}
where
\begin{align}\label{eq:weird}
    c_0 = \frac{1}{\langle 1,1 \rangle_{_M}} \Big( \log \frac{N^2}{M(M,N/q)}
    +4 \Re \frac{L'(1+2iT,\overline{\psi})}{L(1+2iT, \overline{\psi})} + O_{_T}((\log\log (\tfrac{N}{q} +2))^5) \Big),
\end{align}
and the coefficients $c_g$, $c_g'$ satisfy
\begin{align}\label{log23}
    \sum_{g\mid M}(|c_g| + |c_g'|) \ll M^{-1} (\log\log M)^3.
\end{align}
\end{mprop}
 \begin{mrema}One of the pleasant features in Proposition \ref{altE} is that there is no contribution from the newform Eisenstein series with $r > 1$. In addition, by taking $M=N$, Proposition \ref{altE} gives an alternative expression for $\mathcal{E}$ itself.  Finally, we note from Corollary \ref{varphi} that $\varphi^{\scriptscriptstyle{(N)}}_{\infty\infty}(s,\chi)$ vanishes unless $\chi$ is trivial, which means $c_g'=0$ for all $g\mid M$ whenever $\chi$ is nontrivial. 
 \end{mrema}
\begin{proof}
Taking the limit $\beta\rightarrow 0^+$ in \eqref{7dot1} with help of Propositions \ref{Selberg} and \ref{cuspLaurent}, we have 
\begin{multline*}
    \Tr^N_M \mathcal{E} = c_0 + \sum_{g\mid M}c_g G|_g
    + \sum_{1 < r^2 \mid N}\sideset{}{^*}\sum_{\eta (r)} \sum_{g \mid Nr^{-2}} c_{\eta, g}E_{\eta,\eta}(gz,1) \\
    + 2\Re \Big( \varphi^{\scriptscriptstyle{(N)}}_{\infty\infty}(\tfrac{1}{2}+iT,\chi) E_{\infty}^{\scriptscriptstyle{(M)}}(z,1-2iT) \Big),
\end{multline*}
where
\begin{multline*}
    c_0 = c_{\infty,0} + \sum_{\mathfrak{a}\in \mathcal{C}_{\chi}(N)} |\varphi_{\infty\mathfrak{a}}^{\scriptscriptstyle{(N)}}(\tfrac{1}{2}+iT,\chi)|^2 c_{\mathfrak{a},0} \\
    - \frac{1}{\langle 1,1 \rangle_{_M}} \sum_{\mathfrak{a}\in \mathcal{C}_{\chi}(N)}|\varphi_{\infty\mathfrak{a}}^{\scriptscriptstyle{(N)}}(\tfrac{1}{2}+iT,\chi)|^2 \Big( \frac{\varphi'_{\infty\mathfrak{a}}(\tfrac{1}{2}-iT, \overline{\chi})}{\varphi_{\infty\mathfrak{a}}(\tfrac{1}{2}-iT, \overline{\chi})} + \log W^M_N(\mathfrak{a}) \Big),
\end{multline*}
\begin{equation*}
    c_g = c_{\infty,g} + \sum_{\mathfrak{a}\in \mathcal{C}_{\chi}(N)} |\varphi_{\infty\mathfrak{a}}^{\scriptscriptstyle{(N)}}(\tfrac{1}{2}+iT,\chi)|^2 c_{\mathfrak{a},g},
\end{equation*}
and
\begin{equation*}
    c_{\eta, g} = \sum_{\mathfrak{a}\in \mathcal{C}_{\chi}(N)}|\varphi_{\infty\mathfrak{a}}^{\scriptscriptstyle{(N)}}(\tfrac{1}{2}+iT,\chi)|^2 \overline{\eta}(u_{\mathfrak{a}})  c_{\mathfrak{a},\eta, g}.
\end{equation*}
For clarity, we remark that the coefficients $c_{\mathfrak{a},0}$ and $c_{\mathfrak{a},g}$ correspond to the notation from Proposition \ref{cuspLaurent}, but on level $M$.
To simplify, first observe that when $\eta$ (mod {$r$}) is primitive with $r > 1$, then $c_{\eta, g}=0$ for all $g\mid M$. This holds because for each fixed $f \mid N$, $\mathcal{C}_{\chi}(N)$ contains all cusps $\tfrac{u}{f}$ with $u \in \big( \mathbb{Z}/(f,N/f)\mathbb{Z} \big)^{\times}$. Then, since $|\varphi_{\infty\mathfrak{a}}^{\scriptscriptstyle{(N)}}(\tfrac{1}{2}+iT,\chi)|^2$ and $c_{\mathfrak{a},\eta, g}$ are independent of $u_{\mathfrak{a}}$, the sum over $u_{\mathfrak{a}}$ vanishes. 

Next we simplify $c_0$.  By Lemmas \ref{relative} and \ref{width}, Corollary \ref{boundsofcuspLaurent}, and  Remark \ref{ky}, we have 
$\log W^M_N(\mathfrak{a}) = \log W^1_N(\mathfrak{a}) - \log(\frac{M}{(M,(M,f)^2)})$, and so 
\begin{multline*}
 \Vol(Y_0(M)) c_0 = - \log{M}
  + \sum_{\mathfrak{a} \in \mathcal{C}_{\chi}(N)} 
 |\varphi_{\infty\mathfrak{a}}^{\scriptscriptstyle{(N)}}(\tfrac{1}{2}+iT,\chi)|^2 
 \Big(
- \log(f,M)
 \\
 -    \frac{\varphi'_{\infty\mathfrak{a}}(\tfrac{1}{2}-iT, \overline{\chi})}{\varphi_{\infty\mathfrak{a}}(\tfrac{1}{2}-iT, \overline{\chi})} - \log W^1_N(\mathfrak{a})  \Big) + O(\log\log M).
\end{multline*}

Next we apply some approximations to simplify this further.  From Corollary \ref{varphi}, we see that $\varphi_{\infty \mathfrak{a}}(s, \chi) = 0$ unless $f|\frac{N}{q}$, and hence only terms with $(M,f) \mid (M,N/q)$ are in the sum.  Moreover, we have $\frac{(M,N/q)}{(M,f)} \mid \frac{N/q}{f}$.
By \eqref{hardpart1},  we can replace $\log{(f,M)}$ by $\log{(M,N/q)}$ with an acceptable error term, which gives the claimed estimation
\eqref{eq:weird}
for $c_0$.

The estimation of $\sum_{g\mid M}|c_g|$ comes from Corollary \ref{boundsofcuspLaurent} and the fact that
\begin{equation*}
\sum_{g\mid M}|c_g| \leq \sum_{g\mid M} |c_{\infty,g}| + \sum_{\mathfrak{a}}|\varphi_{\infty\mathfrak{a}}(\tfrac{1}{2}+iT,\chi)|^2 \sum_{g\mid M} |c_{\mathfrak{a},g}|.
\end{equation*}

For fixed $T \neq 0$, we have
\begin{align*}
E_{\infty}^{\scriptscriptstyle{(M)}}(z,1+2iT) &= M^{-1-2iT} \frac{\zeta(2+4iT)}{L(2+4iT,\chi_{_0,_M})} \sum_{g\mid M} \frac{\mu(M/g)}{(M/g)^{1+2iT}} E(gz,1+2iT) \\
&= \sum_{g\mid M} c_g'E(gz,1+2iT),
\end{align*}
with $c_g' = \mu(M/g) M^{-2-2iT} g^{1+iT} \frac{\zeta(2+4iT)}{L(2+4iT,\chi_{_0,_M})}$.
It is obvious that $|c_g'| \leq |c_g|$, so the bound of $\sum_g |c_g|$ applies to $\sum_g |c_g'|$.
\end{proof}

\subsection{Proof of \eqref{eq:mainthemMainTerm} and \eqref{eq:alphaphibound}}
Recalling Proposition \ref{altE}, we have
\begin{align*}
\langle \mathcal{E}, \phi \rangle_{_N}=  \langle c_0, \phi \rangle_{_M} + \sum_{g\mid M}c_g \langle G|_g, \phi \rangle_{_M}
+\sum_{g|M} c_g'  \langle E(g \cdot, 1 + 2iT), \phi \rangle_{_M}
,
\end{align*}
where $c_g$ and $c_g'$ are the constants from Proposition \ref{altE}. Define
\begin{align}\label{alpha}
    \alpha_{\phi}=  \sum_{g\mid M}c_g \langle G|_g, \phi \rangle_{_M}
+    
  \sum_{g|M} c_g' \langle E(g\cdot, 1 + 2iT), \phi \rangle_{_M}.
\end{align}
By Lemma \ref{weightedavg} and Remark \ref{widtheq}, we have
\begin{align*}
    \langle \Tr^N_M \mathcal{E}, \phi \rangle_{_M} = c_0 \langle 1, \phi \rangle_{_M} + \alpha_{\phi}.
\end{align*}
Then (\ref{eq:weird}) gives \eqref{eq:mainthemMainTerm}, and \eqref{eq:alphaphibound} follows from \eqref{log23} and (\ref{alpha}). 

\begin{mrema}\label{T=0}
Now consider the case $M=q=1$, and let $T\rightarrow 0$. In this case, we have by \eqref{generallaurent} that
\begin{multline*}
    \varphi_{\infty\infty}(\tfrac{1}{2}+iT) E_{\infty}(z,1-2iT)\\
    = \Big( \varphi_{\infty\infty}(\tfrac{1}{2}) + \varphi_{\infty\infty}'(\tfrac{1}{2}) \cdot (iT) \Big) \cdot \frac{1}{\mathrm{Vol} Y_0(N)} \cdot \frac{1}{-2iT} + 
    \varphi_{\infty\infty}(\tfrac{1}{2}) \cdot G_{\infty}(z) + O(|T|),
\end{multline*}
for some $\Gamma_0(N)$-invariant function $G_{\infty}$ (the constant term in the Laurent expansion of $E^{\scriptscriptstyle{(N)}}_{\infty}(z,s)$ around $s=1$), and so
\begin{multline}\label{2Re}
    2\Re\Big(\varphi_{\infty\infty}(\tfrac{1}{2}+iT) E_{\infty}(z,1-2iT)\Big) = \varphi_{\infty\infty}(\tfrac{1}{2}+iT) E_{\infty}(z,1-2iT) \\
    + \varphi_{\infty\infty}(\tfrac{1}{2}-iT) E_{\infty}(z,1+2iT)
    = - \frac{\varphi_{\infty\infty}'(\tfrac{1}{2})}{\mathrm{Vol} Y_0(N)} + 2\varphi_{\infty\infty}(\tfrac{1}{2}) \cdot G_{\infty}(z) + O(|T|).
\end{multline}
Since $\varphi_{\infty\infty}(\tfrac{1}{2})=-1$ and $\varphi_{\infty\mathfrak{a}}(\tfrac{1}{2})=0$ for all $\mathfrak{a} \in \mathcal{C}(N), \mathfrak{a}\neq \infty$, we have by \eqref{7dot1} and \eqref{2Re}
\begin{align*}
    \langle \mathcal{E}, \phi \rangle_{_N} = \frac{1}{\mathrm{Vol}Y_0(N)}\Big( \varphi_{\infty\infty}(\frac{1}{2})\varphi'(\frac{1}{2}) - \varphi_{\infty\infty}'(\frac{1}{2}) \Big) + O(|T|).
\end{align*}
Letting $T\rightarrow 0$, we see $\langle \mathcal{E}, \phi \rangle_{_N}\rightarrow 0$ as desired.
\end{mrema}

\subsection{Limitations to QUE (continued)}\label{non}
Here we provide the additional details of the example discussed in Section \ref{section:mainterm}.  
Recall in the example that $\chi$ is primitive $\mymod{N}$ and $M$ is a prime divisor of $N$. Then by Proposition \ref{id}, Remark \ref{widtheq}, Lemma \ref{relative} and Corollary \ref{pitt}, we have
\begin{align*}
    \Tr^N_M \mathcal{E} = 
    \lim_{\beta\rightarrow 0^+} \Big( E_{\infty}^{\scriptscriptstyle{(M)}}(z,1+\beta) +
    \big(\tfrac{N}{M}\big)^{\beta} \varphi^{\scriptscriptstyle{(N)}}_{\infty 0}(\tfrac{1}{2}+iT,\chi)\varphi^{\scriptscriptstyle{(N)}}_{\infty 0}(\tfrac{1}{2}+\beta-iT,\overline{\chi})  E_{0}^{\scriptscriptstyle{(M)}}(z,1-\beta) \Big).
\end{align*}
Next, Theorem \ref{c2c} says
\begin{align*}
E_{\infty}^{\scriptscriptstyle{(M)}}(z,1+\beta) = M^{-1-\beta} \frac{\zeta(2+2\beta)}{L(2+2\beta, \chi_{_0,_M})} \Big(E(Mz,1+\beta) - M^{-1-\beta}E(z,1+\beta) \Big),
\end{align*}
and
\begin{align*}
E_{0}^{\scriptscriptstyle{(M)}}(z,1-\beta)  = M^{-1+\beta} \frac{\zeta(2-2\beta)}{L(2-2\beta, \chi_{_0,_M})} \Big( E(z,1-\beta) - M^{-1+\beta} E(Mz, 1-\beta) \Big).
\end{align*}
Then since $\langle \mathcal{E}, \phi \rangle_{_N} = \langle \Tr^N_M \mathcal{E}, \phi \rangle_{_M}$, by Proposition \ref{Laurent} we obtain (\ref{1.8.1}) with 
\begin{align}\label{c1cM}
    c_1= c_M = \frac{\zeta(2)}{L(2,\chi_{_0,_M})} \Big( M^{-1} - M^{-2} \Big) = M^{-1} + O(M^{-2}).
\end{align}
The estimation \eqref{1.8.2} of $c_0$ is contained in \eqref{eq:weird}.

\subsection{Proof of Theorem \ref{portion}} \label{agreement}
Recall (\ref{1.8.1}) for $M$ prime.  
It suffices to show that there are at least $\delta M$ choices of $j$ so that $G$ and $G|_{_M}$ are both bounded (uniformly in $M$) on the support of $\phi = \phi_j^{\scriptscriptstyle{(M)}}$, since then $\langle G, \phi \rangle_{_M}$ and $\langle G|_{_M}, \phi \rangle_{_M}$ are bounded by $O(\norm{ \phi_{_0}}_{_1})$, as desired.

Let $M \geq 1$ be an integer, and consider the following sets.  Let $\mathcal{D}$ denote the standard fundamental domain for $SL_2(\mathbb{Z})$, and let 
\begin{equation}
\mathcal{B}_M = \{ x+iy : x \in \mathbb{R}, \quad M^{-1} < y \leq 20000 M^{-1} \}.
\end{equation} 
For $R > 1$, let $\mathcal{D}(R) = \{ z \in \mathcal{D} : \text{Im}(z) \geq R \}$, and let $\mathcal{D}^c(R) = \mathcal{D} \setminus \mathcal{D}(R)$.
\begin{mrema}
We point out that the only distinct points in $\mathcal{B}_M$ that are $\Gamma_0(M)$-equivalent are integer translates of each other.  This follows because if $z \in \mathcal{B}_M$, then $|cz+d| > 1$ for all coprime integers $c,d$ with $M|c$, $c \neq 0$.  That is, if $z \in \mathcal{B}_M$, and $\gamma \in \Gamma_0(M)$, then  $\text{Im}(\gamma z) < \text{Im}(z)$ unless $c=0$. In fact, the set $\{ z\in \mathbb{H} \mid 0<x<1, |cz+d|>1, (\begin{smallmatrix} *&*\\ c&d \end{smallmatrix})\in \Gamma_0(M)\}$ consists of the interior points of a fundamental domain of $\Gamma_0(M)$, known as the \textit{Ford domain}.
\end{mrema}
\begin{mlemma}\label{+portion}
There exists an absolute constant $\delta_0 > 0$ so that for all $M$ large, there exists at least $\delta_0 M$ $\Gamma_0(M)$-inequivalent coset representatives $\Gamma_0(M) \gamma$ so that $\gamma(\mathcal{D}^c(100)) \subset \mathcal{B}_M$.
 \end{mlemma}
\begin{proof}[Proof of Theorem \ref{portion}] Note that $G$ is bounded on $\mathcal{D}^c(100)$ and hence on $\gamma(\mathcal{D}^c(100))$, for any $\gamma \in SL_2(\mathbb{Z})$.  Meanwhile, $G|_M$ is bounded on $\mathcal{B}_M$, so both $G$ and $G|_M$ are bounded on $\gamma(\mathcal{D}^c(100))$. 
Thus, the test functions $\phi_j^{\scriptscriptstyle{(M)}}$ corresponding to these $\delta_0 M$ coset representatives satisfy the QUE conjecture on shrinking sets, as stated in Theorem \ref{portion}.
\end{proof}

\begin{proof}[Proof of Lemma \ref{+portion}]
We proceed with an explicit construction.  Firstly, we point out that if $\Gamma_0(M) \gamma_1 = \Gamma_0(M) \gamma_2$ with $\gamma_i = (\begin{smallmatrix} * & * \\ c_i & d_i \end{smallmatrix})$, and with $c_1, c_2 \geq 0$, then $\gamma_1 \gamma_2^{-1} = (\begin{smallmatrix} * & * \\ c_1 d_2 - c_2 d_1 & * \end{smallmatrix}) \in \Gamma_0(M)$, and hence $c_1 d_2 \equiv c_2 d_1 \pmod{M}$.  If $-M < c_1 d_2 - c_2 d_1 < M$, then the congruence is an equality, forcing $c_1 = c_2$ and $d_1 = d_2$.  %Consequently, if $c,d$ range over a set of non-negative integers so that $0 \leq cd < M$, then the cosets $\Gamma_0(M) (\begin{smallmatrix} * & * \\ c & d \end{smallmatrix})$ are distinct.  
With this observation, we take the following set: 
\begin{equation*}
S= \{(c,d) \in \mathbb{Z}^2 \mid \tfrac{\sqrt{M}}{100} \leq c \leq \tfrac{\sqrt{M}}{20}, 0 \leq d \leq \tfrac{c}{4} , (c,d) = 1\}.
\end{equation*}
By the discussion above, the cosets $\Gamma_0(M) (\begin{smallmatrix} * & * \\ c & d \end{smallmatrix})$, with $(c,d) \in S$, are distinct.

We claim that for $\gamma = (\begin{smallmatrix} * & * \\ c & d \end{smallmatrix}) \in SL_2(\mathbb{Z})$ with $(c,d) \in S$, then $\gamma(\mathcal{D}^c(100)) \subset \mathcal{B}_M$, and we now proceed to prove this claim.  First we observe that if $z \in \mathcal{D}^c(100)$, then 
\begin{equation*}
\text{Im}(\gamma z) = \frac{y}{(cx+d)^2 + c^2 y^2} \leq 
\frac{y}{c^2 y^2} \leq \frac{2}{\sqrt{3} c^2} \leq \frac{20000}{M},
\end{equation*}
since $c \geq \frac{\sqrt{M}}{100}$. This gives the desired upper bound on the imaginary part.  For the lower bound, we have
\begin{equation*}
 \text{Im}(\gamma z) \geq \frac{y}{(\frac{c}{2} + d)^2 + c^2 y^2} \geq \frac{y}{(3c/4)^2 + c^2 y^2},
\end{equation*}
using $(c/2 + d)^2 \leq (3c/4)^2$.  It is not hard to check that $h(y) = \frac{y}{(3c/4)^2 + c^2 y^2}$ is decreasing in $y$ for $y \geq 3/4$, so the above lower bound on $\text{Im}(\gamma z)$ is minimized when $y=100$.  Thus $\Im(\gamma z) \geq \alpha c^{-2}$ with $\alpha = \frac{100}{100^2 + (3/4)^2}$.  Using $c^2 \leq M/20^2$, we obtain $\Im(\gamma z) \geq 20^2 \alpha/M$.  Checking $20^2 \alpha > 3.999 > 1$ finishes the proof of the desired lower bound on the imaginary part.
 It is easy to check by standard methods that $\# S \sim \delta_0 M$, for some $\delta_0 > 0$. 
\end{proof}

\subsection{Comparison of main terms}\label{truelies}

An astute reader may notice an apparent inconsistency between the main terms displayed in Theorems \ref{secondary} and \ref{main}, and we devote this section to compare these main terms and resolve this paradox.
Recall that
Theorem \ref{main} estimates $\langle |E|^2, \phi \rangle_{_N}$, where $\phi = \phi^{\scriptscriptstyle{(M)}}_j$ is chosen from the system described in Convention \ref{shrinkage}.  One can recover Theorem \ref{secondary} in two different ways from Theorem \ref{main}; the first way is to simply take $M=1$ in Theorem \ref{main}, which visibly reduces to Theorem \ref{secondary}, and the second is to form $\phi_{_0}$ as the sum of $\phi_j^{\scriptscriptstyle{(M)}}$.  
That is, summing over $\phi = \phi_j^{\scriptscriptstyle{(M)}}$ for $j=1,2,...,\nu(M)$, we have
\begin{align*}
\sum_{\phi} \langle |E|^2, \phi \rangle_{_N} = \langle |E|^2, \phi_{_0} \rangle_{_N} &\sim \sum_{\phi} \frac{\langle 1,\phi \rangle_{_M}}{\langle 1,1 \rangle_{_M}} \Big( \log \frac{N^2}{M(M,N/q)} + 4\Re \frac{L'}{L}(1+2iT,\overline{\psi}) \Big) + \sum_{\phi} \alpha_{\phi} \\
&= \frac{\langle 1,\phi_{_0} \rangle_{_1}}{\langle 1,1 \rangle_{_1}} \Big( \log \frac{N^2}{M(M,N/q)} + 4\Re \frac{L'}{L}(1+2iT,\overline{\psi}) \Big) + \sum_{\phi} \alpha_{\phi}.
\end{align*}
This expression has a different shape than that from Theorem \ref{secondary}, which says
\begin{align*}
\langle |E|^2, \phi_{_0} \rangle_{_N} \sim \frac{\langle 1,\phi_{_0} \rangle_{_1}}{\langle 1,1 \rangle_{_1}} \Big( \log N^2 + 4\Re \frac{L'}{L}(1+2iT,\overline{\psi}) \Big).
\end{align*}
For consistency, we must have
\begin{align}\label{consistencycheck}
\sum_{\phi} \alpha_{\phi} \sim \frac{\langle 1,\phi_{_0} \rangle_{_1}}{\langle 1,1 \rangle_{_1}} \log (M(M,N/q)).
\end{align}
We wish to check this directly, at least in some special cases.  For simplicity of exposition, we take $q=N$ (i.e., $\chi$ is primitive), and $M$ prime.

In (\ref{alpha}), we have $c_g' = 0$ since $q \neq 1$, 
whence
\begin{align*}
\sum_{\phi} \alpha_{\phi} = \sum_{\phi} \sum_{g\mid M} c_g \langle G|_g, \phi \rangle_{_M} = \sum_{g\mid M} c_g \langle G|_g, \phi_{_0} \rangle_{_M}.
\end{align*}
Since $\phi_{_0}$ is $SL_2(\mathbb{Z})$-invariant, we have
\begin{align*}
\sum_{g\mid M} c_g \langle G|_g, \phi_{_0} \rangle_{_M} = \sum_{g\mid M} c_g \frac{\nu(M)}{\nu(g)} \langle \Tr^g_1(G|_g) , \phi_{_0}\rangle_{_1} .
\end{align*}
On the other hand, one can check directly (see \cite[Sections 5.1--5.4]{DS}) that
\begin{align*}
\Tr^g_1 (f|_g) = \sqrt{g} T_g (f),
\end{align*}
for any automorphic function $f$ of level $1$.
Hence by Lemma \ref{TnG} and (\ref{c1cM}), 
\begin{align*}
\sum_{\phi} \alpha_{\phi} &= \sum_{g\mid M} c_g \frac{\nu(M)}{\nu(g)} \sqrt{g} \langle T_g(G), \phi_{_0} \rangle_{_1} \\
&= \sum_{g\mid M} c_g \frac{\nu(M)}{\nu(g)} \sqrt{g} \Big( \lambda(g) \langle G, \phi_{_0} \rangle_{_1} + \frac{3}{\pi} \sqrt{g}\Big(\sum_{a\mid g} a^{-1}\log \frac{g}{a^2}\Big) \langle 1,\phi_{_0} \rangle_{_1} \Big) \\
&= \frac{\langle 1, \phi_{_0} \rangle_{_1}}{\langle 1,1 \rangle_{_1}} (\log{M}) (1+O(M^{-1}))  + \langle G, \phi_{_0} \rangle_{_1}(2 + O(M^{-1})),
\end{align*}
which indeed agrees with \eqref{consistencycheck}. 

\section{QUE for Eisenstein series attached to other cusps}\label{othercusps}
This section concentrates on proving Theorem \ref{thirdary}. 
Assume $\chi$ is primitive modulo $N$ throughout this section. By Proposition \ref{criterion}, $\mathcal{C}_{\chi}(N)$ consists of Atkin--Lehner cusps. Recall for a cusp $\mathfrak{a}=\frac{1}{f}\in \mathcal{C}_{\chi}(N)$, we denote the cusp $\frac{1}{N/f}\in \mathcal{C}_{\chi}(N)$ by $\mathfrak{a}^*$ and call it the Atkin--Lehner conjugate of $\mathfrak{a}$. It is easy to see by Lemma \ref{width} that $W_{\mathfrak{a}}=N/f$, and $W_{\mathfrak{a}^*}=f$.

% with the following paraphrase.
% \begin{theon}[\ref{thirdary} (restated)]\label{othercusp}
% For $\chi$ primitive $\medspace\shortmod{N}$ and $\mathfrak{a}\in \mathcal{C}_{\chi}(N)$, we have the same formula as in Theorem \ref{secondary} for $E=E_{\mathfrak{a}}(z,s,\chi)$ with fixed $s=\frac{1}{2}+iT$.
% \end{theon}

\subsection{Identification of $\mathcal{E}$}
Corollary \ref{pitt} and Proposition \ref{nonsingular} give the cuspidal behavior of $|E_{\mathfrak{a}}|^2$ at any $\mathfrak{b}\in \mathcal{C}(N)$.  The following proposition can be proved similarly as Proposition \ref{id}.
\begin{mprop}\label{id'}
For $E=E_{\mathfrak{a}}(z,\frac{1}{2}+iT,\chi)$ as in Theorem \ref{thirdary}, we have $|E|^2-\mathcal{E}\in \mathcal{B}_{\varepsilon}(Y_0(N))$ for arbitrarily small $\varepsilon >0$ with
\begin{align*}
\mathcal{E} = \lim_{\beta\rightarrow 0^+} \Big( E_{\mathfrak{a}}(z,1+\beta) + \varphi_{\mathfrak{aa}^*}(\tfrac{1}{2}+iT,\chi) \varphi_{\mathfrak{aa}^*}(\tfrac{1}{2}+\beta-iT,\overline{\chi}) E_{\mathfrak{a}^*}(z,1-\beta) \Big).
\end{align*}
\end{mprop}
The following subsections deal with $\langle |E|^2-\mathcal{E}, \phi_{_0}\rangle_{_N}$ and $\langle\mathcal{E}, \phi_{_0} \rangle_{_N}$ separately.

\subsection{Error term}
Since $|E|^2-\mathcal{E}\in \mathcal{B}_{\varepsilon}(Y_0(N))$ and $M=1$, the analog of  (\ref{reget}) is
\begin{align*}
    \langle |E|^2-\mathcal{E}, \phi_{_0}\rangle_{_N} = \sum_{j\geq 1} \langle |E|^2,u_j\rangle_{_N} \langle u_j, \phi_{_0} \rangle_{_1} 
+  \frac{1}{4\pi}\int_{-\infty}
^{\infty} \langle |E|^2, E(\cdot,\tfrac{1}{2}+it)\rangle_{_N}^{\text{reg}} \langle E(\cdot,\tfrac{1}{2}+it), \phi_{_0} \rangle_{_1} dt. \end{align*}

Recall from (\ref{expansionq=N}) that $E_{\mathfrak{a}}(z,s,\chi) = N^{-s} E_{\chi_1, \chi_2}(z,s)$, where $\chi = \chi_1 \overline{\chi_2}$ with $\chi_1$ modulo $N/f$ and $\chi_2$ modulo $f$.
As a result, with (\ref{uj}), (\ref{YW}) and (\ref{specialq=N}) we have for some $\epsilon$ with $|\epsilon| = 1$
\begin{align*}
\langle |E_{\mathfrak{a}}(\cdot,s,\chi)|^2, u_j \rangle_{_N} &= \chi_1(-1) N^{-\overline{s}} \int_0^1 \int_0^{\infty} y^{s-2} \overline{E_{\chi_1, \chi_2}|_{\sigma_{\mathfrak{a}}}}   \overline{u_j}(\sigma_{\mathfrak{a}}z) dxdy \\
&=  \epsilon N^{-\overline{s}}\int_0^1 \int_0^{\infty} y^{s-2} \overline{ E_{1, \chi_1\chi_2}}  \overline{u_j}(\tfrac{N}{f}z ) dxdy \\
&= \frac{2 \epsilon F_T(t_j)}{N^{\overline{s}}(2\pi)^s \theta_{1,\chi_1\chi_2}} (\overline{\lambda_{1,\chi_1\chi_2}}(-1)+ \overline{\lambda_j}(-1))   \sum_{n\geq 1} \frac{\overline{\lambda_{1,\chi_1\chi_2}}(\frac{N}{f}n,s)  \lambda_j(n)}{n^s}.
\end{align*}
Then we can meromorphically continue the above equation to the whole complex plane, and take $s=\frac{1}{2}+iT$, where the Dirichlet series equals a finite Euler product of size $O(N^{\varepsilon})$ times
\begin{align*}
\frac{L(\tfrac{1}{2},u_j) L(\tfrac{1}{2}+2iT, u_j \otimes \chi_1\chi_2)}{L(1+2iT,\chi_1\chi_2)},
\end{align*}
which has Burgess bound $N^{\frac{3}{8}+\varepsilon}$.  Hence, in total we have
\begin{align*}
\langle |E_{\mathfrak{a}}(\cdot,s,\chi)|^2, u_j \rangle_{_N} \ll_T  e^{\frac{\pi}{2}H_T(t_j)}N^{-\frac{1}{8}+\varepsilon},
\end{align*}
for the same $H_T(t_j)$ as in (\ref{HT}). Mimicking the proof of Proposition \ref{ET1}, we have
\begin{align*}
\sum_{u\in \mathcal{O}(1)} \langle |E_{\mathfrak{a}}|^2,u \rangle_{_N} \langle u, \phi_{_0} \rangle_{_1} =\sum_{j\geq 1} \langle |E_{\mathfrak{a}}|^2,u_j \rangle_{_N} \langle u_j, \phi_{_0} \rangle_{_1} \ll_T N^{-\frac{1}{8}+\varepsilon} \norm{\phi_{_0}}_{_2},
\end{align*}
and likewise,
\begin{align*}
\frac{1}{4\pi} \int_{-\infty}^{\infty} \langle |E_{\mathfrak{a}}|^2, E(\cdot, \tfrac{1}{2}+it) \rangle_{_N}^{\text{reg}} \langle E(\cdot, \tfrac{1}{2}+it), \phi_{_0} \rangle_{_1} dt  \ll_T N^{-\frac{1}{8}+\varepsilon}\norm{\phi_{_0}}_{_2}.
\end{align*}

\subsection{Main term}
Since $W_{\mathfrak{a}^*}=f$ by Lemma \ref{width}, we can derive from Lemma \ref{relativetrace} and Proposition \ref{id'} that
\begin{multline*}
\langle \mathcal{E}, \phi_{_0} \rangle_{_N} = \langle \Tr^N_1 \mathcal{E}, \phi_{_0} \rangle_{_1} 
= \lim_{\beta\rightarrow 0^+} \Big((\tfrac{N}{f})^{-\beta} \langle E(\cdot, 1+\beta), \phi_{_0} \rangle_{_1} \\
+ \varphi_{\mathfrak{aa}^*}(\tfrac{1}{2}+iT,\chi) \varphi_{\mathfrak{aa}^*}(\tfrac{1}{2}+\beta-iT,\overline{\chi}) f^{\beta} \langle E(\cdot,1-\beta), \phi_{_0} \rangle_{_1} \Big).
\end{multline*}
Substituting the Laurent expansion by Proposition \ref{Laurent}, we have
\begin{align*}
\langle \mathcal{E}, \phi_{_0} \rangle_{_N} = \frac{\langle 1, \phi_{_0} \rangle_{_1}}{\langle 1,1 \rangle_{_1}} \Big(-\log \tfrac{N}{f} - \log f - \frac{ \varphi'_{\mathfrak{aa}^*}}{ \varphi_{\mathfrak{aa}^*}}(\tfrac{1}{2}-iT,\overline{\chi}) \Big)  +2 \langle G, \phi_{_0} \rangle_{_1},
\end{align*}
while from Corollary \ref{pitt} we see that
\begin{align*}
 \frac{ \varphi'_{\mathfrak{aa}^*}}{ \varphi_{\mathfrak{aa}^*}}(\tfrac{1}{2}-iT,\overline{\chi}) = -3\log N - 4\Re \frac{L'}{L}(1+2iT, \overline{\chi_1\chi_2}) + O_T(1).
\end{align*}
After subtraction we arrive at
\begin{align*}
\langle \mathcal{E}, \phi_{_0} \rangle_{_N} = \frac{\langle 1, \phi_{_0} \rangle_{_1}}{\langle 1,1 \rangle_{_1}} \Big(2\log N + 4\Re \frac{L'}{L}(1+2iT, \overline{\chi_1\chi_2}) + O_T(1) \Big)  +2 \langle G, \phi_{_0} \rangle_{_1}.
\end{align*}

\end{document}